\documentclass[11pt]{amsart}
\usepackage{amsmath}
\usepackage{amssymb}
\usepackage{amsfonts}
\usepackage{amsthm}
\usepackage{enumerate}
\usepackage[mathscr]{eucal}
\usepackage{verbatim}
\usepackage{amsthm}
\usepackage{amscd}

\newtheorem{theorem}{Theorem}[section]

\newtheorem{proposition}[theorem]{Proposition}

\newtheorem{corollary}[theorem]{Corollary}
\newtheorem{lemma}[theorem]{Lemma}

\theoremstyle{definition}

\newtheorem{innercustomgeneric}{\customgenericname}
\providecommand{\customgenericname}{}

\newcommand{\newcustomtheorem}[2]{\newenvironment{#1}[1]
  {\renewcommand\customgenericname{#2}
   \renewcommand\theinnercustomgeneric{##1}\innercustomgeneric}{\endinnercustomgeneric}}

\newcustomtheorem{customthm}{Theorem}

\newcommand{\newcustomlemma}[2]{\newenvironment{#1}[1]
  {\renewcommand\customgenericname{#2}
   \renewcommand\theinnercustomgeneric{##1} \innercustomgeneric}{\endinnercustomgeneric}}

\newcustomlemma{customlemma}{Lemma}

\newcustomlemma{customproposition}{Proposition}

\newcommand\relphantom[1]{\mathrel{\phantom{#1}}}

\newcommand{\Ga}{\Gamma}

\newcommand{\rr}{\mathbb{R}}
\newcommand{\nn}{\mathbb{N}}
\newcommand{\mm}{\mathbf{m}}
\newcommand{\rn}{\mathbb{R}^n}
\newcommand{\rd}{\mathbb{R}^d}
\newcommand{\zz}{\mathbb{Z}}

\newcommand{\cc}{\mathbb{C}}
\newcommand{\bb}{\mathbf{b}}
\newcommand{\ff}{\mathbf{f}}

\newcommand{\wt}{\widetilde}
\newcommand{\wh}{\widehat}

\numberwithin{equation}{section}

\begin{document}

\begin{thanks}
{The author  is supported in part by NRF grant 2019R1F1A1044075 and by a KIAS Individual Grant MG070001 at Korea Institute for Advanced Study.}
\end{thanks}

\address{School of Mathematics \\
           Korea Institute for Advanced Study, Seoul\\
           Republic of Korea}
   \email{qkrqowns@kias.re.kr}

\author{Bae Jun Park}

\title[Fourier multipliers on a vector-valued function space]{Fourier multipliers on a vector-valued function space}


\begin{abstract} 
We study multiplier theorems on a vector-valued function space, which is a generalization of the results of Calder\'on and Torchinsky \cite{Ca_To} and Grafakos, He, Honz\'ik, and Nguyen \cite{Gr_He_Ho_Ng}, and an improvement of the result of Triebel \cite{Tr1, Tr}. For $0<p<\infty$ and $0<q\leq \infty$ we obtain that if $r>\frac{d}{s-(d/\min{(1,p,q)}-d)}$, then
$$\big\Vert \big\{\big( m_k \widehat{f_k}\big)^{\vee}\big\}_{k\in\mathbb{Z}}\big\Vert_{L^p(\ell^q)}\lesssim_{p,q} \sup_{l\in\mathbb{Z}}{\big\Vert m_l(2^l\cdot)\big\Vert_{L_s^r(\rd)}} \big\Vert \big\{f_k\big\}_{k\in\mathbb{Z}}\big\Vert_{L^p(\ell^q)}, ~~f_k\in\mathcal{E}(A2^k),$$
under the condition $\max{(|d/p-d/2|,|d/q-d/2|)}<s<d/\min{(1,p,q)}$. An extension to $p=\infty$ will be additionally considered in the scale of Triebel-Lizorkin space.
 Our result is sharp in the sense that the Sobolev space in the above estimate cannot be replaced by Sobolev spaces $L_s^r$ with $r\leq \frac{d}{s-(d/\min{(1,p,q)}-d)}$.
\end{abstract}

\maketitle

\section{{Introduction and main results}}\label{introduction}

Let $S(\mathbb{R}^d)$ denote the Schwartz space and $S'(\mathbb{R}^d)$ the space of tempered distributions. For the Fourier transform of $f\in S(\mathbb{R}^d)$ we use the definition $\widehat{f}(\xi):=\int_{\mathbb{R}^d}{f(x)e^{-2\pi i\langle x,\xi\rangle}}dx$ and denote by $f^{\vee}(\xi):=\wh{f}(-\xi)$ the inverse Fourier transform of $f$. We also extend these transforms to the space of tempered distributions.

For $m\in L^{\infty}(\rd)$  the multiplier operator $T_m$ is defined by $T_mf(x):=\big( m\widehat{f}\big)^{\vee}(x)$ for $f\in S(\rd)$.
The classical Mikhlin multiplier theorem \cite{Mik} states that if a function $m$  satisfies
\begin{eqnarray*}
\big|  \partial_{\xi}^{\beta}m(\xi)  \big|\lesssim_{\beta}|\xi|^{-|\beta|}
\end{eqnarray*} for all multi-indices $\beta$ with $|\beta|\leq \big[d/2\big]+1$, then the operator $T_m$ is bounded in $L^p(\rd)$ for $1<p<\infty$.
 In \cite{Ho} H\"ormander sharpened the result of Mikhlin, using the weaker condition
\begin{eqnarray}\label{hocondition}
\sup_{l\in\zz}{\big\Vert m(2^l\cdot)\widehat{\phi}\big\Vert_{L^2_s(\rd)}}<\infty
\end{eqnarray} for $s>d/2$, where $L^2_s(\rd)$ denotes the standard fractional Sobolev space on $\rd$ and $\phi$ is a Schwartz function on $\rd$, which generates a Littlewood-Paley partition of unity via a dyadic dilation, defined in Section \ref{preliminary}.
Calder\'on and Torchinsky \cite{Ca_To} proved that if (\ref{hocondition}) holds for $s>d/p-d/2$, then $m$ is a Fourier multiplier of Hardy space $H^p(\rd)$ for $0<p\leq 1$. A different proof was given by Taibleson and Weiss \cite{Ta_We}. 
It turns out that the condition $s>d/\min{(1,p)}-d/2$ is optimal for the boundedness to hold and it is natural to ask whether (\ref{hocondition}) can be weakened by replacing $L^{2}_s(\rd)$ by other function spaces. 
Baernstein and Sawyer \cite{Ba_Sa} obtained endpoint $H^p(\rd)$ estimates by using Herz space conditions for $\big( m(2^j\cdot)\widehat{\phi} \big)^{\vee}$ and these estimates were improved and extended to Triebel-Lizorkin spaces by Seeger \cite{Se} and Park \cite{Park2}.
On the other hand, for $1<p<\infty$, using an interpolation method, Calder\'on and Torchinsky \cite{Ca_To} replaced $L_s^2(\rd)$ in (\ref{hocondition}) by $L_s^r(\rd)$ for the $L^p$-boundedness to hold and the assumption in their result was replaced by a weaker one by Grafakos, He, Honz\'ik, and Nguyen \cite{Gr_He_Ho_Ng}.
 Let $(I-\Delta)^{s/2}$ be the inhomogeneous fractional Laplacian operator, explicitly given by
\begin{equation*}
(I-\Delta)^{s/2}f:=\big((1+4\pi^2|\cdot|^2)^{s/2}\wh{f}\big)^{\vee}
\end{equation*} and let $L_s^r(\rd)$ be the space containing tempered distributions $f$, defined on $\rd$, for which the norm 
\begin{equation*}
\Vert f\Vert_{L^r_s(\rd)}:=\big\Vert (I-\Delta)^{s/2}f\big\Vert_{L^r(\rd)}
\end{equation*} is finite.

\begin{customthm}{A}\label{thma} 
Let $1<p<\infty$ and $|d/p-d/2|<s<d$. Suppose that
\begin{equation*}
\sup_{l\in\zz}{\big\Vert m(2^l\cdot )\wh{\phi}\big\Vert_{L_s^r(\rd)}}<\infty \qquad \text{ for }~ r>d/s.
\end{equation*}
Then $T_m$ is bounded in $L^p(\rd)$.
\end{customthm}
We also refer to \cite{Gr_Park, Gr_Sl} for further improvement of the multiplier theorem by using Lorentz space conditions.\\

 A vector-valued version of H\"ormander's multiplier theorem was studied by Triebel \cite{Tr0}, \cite[2.4.9]{Tr}.
 For $r>0$ let $\mathcal{E}(r)$ denote the space of all distributions whose Fourier transform is supported in $\big\{\xi\in\mathbb{R}^d:|\xi|\leq 2r\big\}$.
Let $A>0$. For $0<p<\infty$ and $0<q\leq \infty$ or for $p=q=\infty$ we define 
 \begin{equation*}
L_A^p(\ell^q):=\big\{\{f_k\}_{k\in\mathbb{Z}}\subset S':f_k\in\mathcal{E}(A2^k), \big\Vert \big\{ f_k\big\}_{k\in\mathbb{Z}}\big\Vert_{L^p(\ell^q)}<\infty \big\}.
\end{equation*}
To give a rigorous definition of the space, we recall that for each $f_k\in\mathcal{E}(A2^k)$
\begin{equation*}
f_k=f_k\ast \Pi_k \quad \text{ in the sense of tempered distribution}
\end{equation*}
where $\Pi_k$ is a Schwartz function whose Fourier transform is equal to $1$ on the ball of radius $A2^{k+1}$, centered at $0$ and is supported in a larger ball.
Since convolution between a tempered distribution and a Schwartz function is a smooth function, $f_k\ast \Pi_k$ is actually a smooth function and thus, the norm $\Vert \{f_k\}_{k\in \mathbb{Z}}\Vert_{L^p(\ell^q)}$ can be interpreted as 
\begin{equation*}
\Vert \{f_k\}_{k\in \mathbb{Z}}\Vert_{L^p(\ell^q)}=\Vert \{f_k\ast \Pi_k\}_{k\in \mathbb{Z}}\Vert_{L^p(\ell^q)}.
\end{equation*}
In the rest of this paper, we think of $f_k\in \mathcal{E}(A2^k)$ as a smooth function $f_k\ast \Pi_k$.

Then  $L_A^p(\ell^q)$ is a quasi-Banach space (Banach space if $p,q\geq 1$) with a (quasi-)norm $\Vert \cdot\Vert_{L^p(\ell^q)}$ ( see \cite{Tr} for more details ).

\begin{customthm}{B}\label{thmb} 
Let $0<p<\infty$, $0<q\leq \infty$, and $A>0$. Suppose $f_k\in\mathcal{E}(A2^k)$ for each $k\in\mathbb{N}$, and $\{m_k\}_{k\in\mathbb{N}}$ satisfies
\begin{equation}\label{mkkcondition}
\sup_{l\in\mathbb{N}}{\big\Vert m_l(2^l\cdot)\big\Vert_{L_s^2(\rd)}}<\infty 
\end{equation}
for
\begin{equation*}
 s>\Big\{\begin{array}{ll}
d/\min{(1,p,q)}-d/2\quad & \text{if}\quad q<\infty\\
d/p+d/2 \quad  & \text{if}\quad q=\infty 
\end{array}.
\end{equation*}
Then \begin{equation}\label{vectormulti}
\big\Vert \big\{\big( m_k \widehat{f_k}\big)^{\vee}\big\}_{k\in\mathbb{N}}\big\Vert_{L^p(\ell^q)}\lesssim_{p,q} \sup_{l\in\mathbb{N}}{\big\Vert m_l(2^l\cdot)\big\Vert_{L_s^2}} \big\Vert \big\{f_k\big\}_{k\in\mathbb{N}}\big\Vert_{L^p(\ell^q)}.
\end{equation}
\end{customthm}
It was first proved that if (\ref{mkkcondition}) holds for $s>d/2$, then (\ref{vectormulti}) works for $1<p,q<\infty$,  by using H\"ormander's multiplier theorem. 
For the case $0<p<\infty$ and $0<q\leq \infty$, it is easy to obtain that
(\ref{vectormulti}) is true under the assumption (\ref{mkkcondition}) with $s>d/2+d/\min{(p,q)}$. 
Then a complex interpolation method is applied to derive $s>d/\min{(1,p,q)}-d/2$ for general $0<p,q<\infty$.
However, the method  cannot be applied to the endpoint case $q=\infty$ and thus the assumption $s>d/p+d/2$ is required when $q=\infty$, which is stronger than seemingly ``natural" condition $s>d/\min{(1,p)}-d/2$. \\

The aim of this paper is to provide an improvement of Theorem \ref{thmb}, which would be actually a vector-valued extension of Theorem \ref{thma} in the full range $0<p\leq \infty$.
Let
\begin{equation*}
\tau^{(s,p)}:=\frac{d}{s-(d/\min{(1,p)}-d)}, \qquad \tau^{(s,p,q)}:=\frac{d}{s-(d/\min{(1,p,q)}-d)}.
\end{equation*}
For $\mm:=\{m_k\}_{k\in\mathbb{Z}}$, throughout this work we will use the notation:
\begin{equation*}
\mathcal{L}_s^r[\mm]:=\sup_{l\in\mathbb{Z}}{\big\Vert m_l(2^{l}\cdot)\big\Vert_{L_s^r(\rd)}}.
\end{equation*}

\begin{theorem}\label{main1}
Let $0<p<\infty$ and $0<q\leq \infty$, $A>0$, and
\begin{equation*}
\max{\big( \big|d/p-d/2 \big|, \big|d/q-d/2\big|\big)}<s<d/\min{(1,p,q)}.
\end{equation*}
Suppose $f_k\in\mathcal{E}(A2^k)$ for each $k\in\mathbb{Z}$ 
and $\mm:=\{m_k\}_{k\in\mathbb{Z}}$ satisfies 
\begin{equation*}
\mathcal{L}_s^r[\mm]<\infty \qquad \text{ for } ~ r>\tau^{(s,p,q)}.
\end{equation*}
Then
\begin{equation}\label{mainest1}
\big\Vert \big\{ \big( m_k \widehat{f_{k}}\big)^{\vee}\big\}_{k\in\mathbb{Z}}\big\Vert_{L^p(\ell^q)}\lesssim_{p,q} \mathcal{L}_s^r[\mm] \big\Vert \big\{ f_k\big\}_{k\in\mathbb{Z}}\big\Vert_{L^p(\ell^q)}.
\end{equation}
Moreover, the inequality also holds for $p=q=\infty$.
\end{theorem}

 Theorem \ref{main1} can be extended to the case $p=\infty$ and $0<q<\infty$ in the scale of Triebel-Lizorkin space. To describe this, let $\mathcal{D}$ denote the collection of all dyadic cubes in $\mathbb{R}^d$ and for each $P\in \mathcal{D}$ let $\ell(P)$ be the side length of $P$.

\begin{theorem}\label{main2}
Let $0<q< \infty$, $A>0$, $\mu\in\zz$, and
\begin{equation*}
 \big|d/q-d/2\big|<s<d/\min{(1,q)}.
\end{equation*}
Suppose $f_k\in\mathcal{E}(A2^k)$ for each $k\in\mathbb{Z}$ 
and $\mm:=\{m_k\}_{k\in\mathbb{Z}}$ satisfies
\begin{equation*}
\mathcal{L}_s^r[\mm]<\infty, \qquad \text{ for }~ r>\tau^{(s,q)}.
\end{equation*}
Then
\begin{align*}
&\sup_{P\in\mathcal{D}, \ell(P)\leq 2^{-\mu}}{\Big(\frac{1}{|P|}\int_P{\sum_{k=-\log_2{\ell(P)}}^{\infty}{\big| \big(m_k\widehat{f_{k}} \big)^{\vee}(x)\big|^q}}dx \Big)^{1/q}}\\
&\lesssim_{q}  \mathcal{L}_s^r[\mm]\sup_{P\in\mathcal{D}, \ell(P)\leq 2^{-\mu}}{\Big(\frac{1}{|P|}\int_P{\sum_{k=-\log_2{\ell(P)}}^{\infty}{|f_{k}(x)|^q}}dx \Big)^{1/q}}
\end{align*} uniformly in $\mu$.
\end{theorem}

As a corollary of the two theorems, we can prove the $\dot{F}_{p}^{\alpha,q}$-boundedness of the operator $T_m$, which is a generalization of Theorem \ref{thma} and an improvement of the result in \cite{Tr1}.
\begin{corollary}\label{maincoro}
Let $0<p,q\leq \infty$ and $\alpha \in\rr$. Suppose 
\begin{equation*}
\max{\big( \big|d/p-d/2 \big|, \big|d/q-d/2\big|\big)}<s<d/\min{(1,p,q)}
\end{equation*}
and $m\in L^{\infty}(\rd)$ satisfies 
\begin{equation*}
\sup_{l\in\zz}{\big\Vert m(2^l\cdot)\wh{\phi}\big\Vert_{L^r_s(\rd)}}<\infty \quad \text{ for }~ r>\tau^{(s,p,q)}.
\end{equation*}
Then
\begin{equation*}
\Vert T_mf\Vert_{\dot{F}_p^{\alpha,q}(\rd)}\lesssim \sup_{l\in\zz}{\big\Vert m(2^l\cdot)\wh{\phi}\big\Vert_{L^r_s(\rd)}}\Vert f\Vert_{\dot{F}_{p}^{\alpha,q}(\rd)}.
\end{equation*}

\end{corollary}
This follows from setting $m_k=m\wh{\phi_k}$ and $f_k=2^{\alpha k}\wt{\phi_k}\ast f$ where $\wt{\phi_k}:=\phi_{k-1}+\phi_k+\phi_{k+1}$. 
The detailed proof is omitted as standard arguments are applicable.
We refer the reader to Section \ref{preliminary} for the definition of Triebel-Lizorkin spaces $\dot{F}_{p}^{\alpha,q}(\rd)$. As the space $\dot{F}_p^{\alpha,q}$ is a generalization of many function spaces such as Lebesgue space, Hardy space and $BMO$, Corollary \ref{maincoro} also implies the boundedness of $T_m$ on such function spaces.\\

It turns out that the condition $s>|d/p-d/2|$ is optimal for the $L^p$-boundedness to hold in Theorem \ref{thma} and the proof can be found in Slav\'ikov\'a \cite{Sl}. Moreover, Grafakos and Park \cite{Gr_Park} recently proved that the condition $r>d/s$ should be also necessary in the theorem, using properties of Bessel potentials, which will be described in (\ref{esth}) later. We now consider the sharpness of the condition $r>\tau^{(s,p,q)}$ in Theorem \ref{main1}. Our claim is that (\ref{mainest1}) fails for $r=\tau^{(s,p,q)}$.
\begin{theorem}\label{negativemain}
Let $0<p<\infty$, $0<q\leq \infty$, and $d/\min{(1,p,q)}-d<s<d/\min{(1,p,q)}.$
 Then there exists $\mm:=\{m_k\}_{k\in\zz}$ such that $\mathcal{L}_{s}^{\tau^{(s,p,q)}}[\mm]<\infty$, but (\ref{mainest1}) does not hold.
\end{theorem}
Remark that the assumption $d/\min{(1,p,q)}-d<s<d/\min{(1,p,q)}$ is clearly weaker than $\max{\big( \big|d/p-d/2 \big|, \big|d/q-d/2\big|\big)}<s<d/\min{(1,p,q)}$ in Theorem \ref{main1}.\\

We first study Theorem \ref{main2}, using a proper separation of $f_k$ and $F_{\infty}$-variants of Peetre's maximal inequality, introduced by the author \cite{Park1}. 
For the proof of Theorem \ref{main1}, 
the case $0<p=q\leq \infty$ can be handled in a easy way via the $L^p$-boundedness of $T_{m_k}$, which is stated in Lemma \ref{basiclemma}, and thus our interest will be given to the case $p\not= q$. For the case $0<p\leq 1$ and $p<q\leq \infty$ we will establish a discrete characterization of $L_A^p(\ell^q)$ by using the $\varphi$-transform of Frazier and Jawerth \cite{Fr_Ja0, Fr_Ja1, Fr_Ja2, Fr_Ja3} and apply atomic decomposition of discrete function space $\dot{f}_{p}^{0,q}$ in \cite{Fr_Ja3}, which is analogous to the atomic decomposition of $H^p(\rd)$.
When $0<q\leq 1$ and $q<p<\infty$, the proof relies on a characterization of $L_A^p(\ell^q)$ by a dyadic version of the Fefferman-Stein sharp maximal function \cite{Fe_St1}.
The remaining case $1<p<\infty$ and $1<q\leq \infty$ follows from a combination of complex interpolation techniques in Proposition \ref{complexinterpolation} and duality arguments in Lemma \ref{dualitylemma}. 
The central idea to prove Theorem \ref{negativemain} is a necessary condition for a vector-valued inequality of convolution operator in the paper of Christ and Seeger \cite{Ch_Se} and a behavior of variants of Bessel potentials in the paper of Grafakos and Park \cite{Gr_Park}. See (\ref{qbound}) and (\ref{esth}) below.\\

{\bf Basic setting :}
 The constant $A$ plays a minor role in the results and in fact, it affects the results only up to a constant. Hence, we fix $A=2^{-2}$ in the proof to avoid unnecessary complications.
 Moreover, if $f_k\in\mathcal{E}(2^{k-2})$, then $\big( m_k\widehat{f_{k}}\big)^{\vee}=\big( (m_k\widehat{\Psi_{k}})\widehat{f_k}\big)^{\vee}$ 
 where $\Psi_k\in S(\rd)$ is a Schwartz function having the properties that $Supp(\wh{\Psi_k})\subset \{\xi\in\rd: |\xi|\leq 2^k\}$ and $\wh{\Psi_k}(\xi)=1$ for $|\xi|\leq 2^{k-1}$. This function will be officially defined in Section \ref{phitransform}, using dyadic dilation $\Psi_k(x)=2^{kd}\Psi_0(2^kx)$.
 Then the Kato-Ponce inequality \cite{Ka_Po} yields that for $1<r<\infty$ and $s\geq 0$,
 \begin{equation*}
 \big\Vert \big(m_k\widehat{\Psi_k}\big)(2^{k}\cdot)\big\Vert_{L_s^{r}(\rd)}=\big\Vert m_k(2^k\cdot)\widehat{\Psi_0} \big\Vert_{L_s^{r}(\rd)}\lesssim \Vert m_k(2^k\cdot)\Vert_{L_s^{r}(\rd)}
 \end{equation*} and this enables us to assume that 
\begin{equation}\label{assumptionm}
Supp(m_k)\subset \big\{ \xi\in\mathbb{R}^d: |\xi|\leq 2^{k}\big\}
\end{equation} in the proof.
With this assumption, we can write $\big(m_k\widehat{f_{k}}\big)^{\vee}(x)=m_k^{\vee}\ast f_{k}(x)$.\\

This paper is organized as follows. Section \ref{preliminary} is dedicated to preliminaries, introducing definitions and general properties which will be used in our proofs.
Two characterizations of $L^p_A(\ell^q)$ will be given in Section \ref{phitransform}, and by using one of them we dualize the function space $L^p_A(\ell^q)$ for $1<p<\infty$ and $1\leq q<\infty$ in Section \ref{dualization}. In Section \ref{complexsection} we present a complex interpolation theorem for multipliers on $L^p_A(\ell^q)$, based on the idea of Triebel \cite[2.4.9]{Tr}. Section \ref{keylemmasection} contains a lemma which will play a fundamental role in the proof of both Theorem \ref{main1} and \ref{main2}.
The proof of Theorem \ref{main1}, \ref{main2}, and \ref{negativemain} will be provided in the last three sections.\\

{\bf Notations :} 
We use standard notations. Let $\mathbb{N}$ be the collection of all natural numbers and $\mathbb{N}_0:=\mathbb{N}\cup\{0\}$. 
Denote by $\mathbb{Z}$ and $\rr$ the set of all integers and the set of all real numbers, respectively.
 Let  $\mathcal{D}$ stand for the set of all dyadic cubes in $\mathbb{R}^d$ as above and
 for each $k\in\mathbb{Z}$, let $\mathcal{D}_{k}$ be the subset of $\mathcal{D}$ consisting of the cubes with side length $2^{-k}$.
 We use the symbol $X\lesssim Y$ to indicate that $X\leq CY$ for some constant $C>0$, possibly different at each occurrence,  and  $X\approx Y$ if $A\lesssim B$ and $B\lesssim A$ simultaneously.\\

\section{Preliminaries}\label{preliminary}

\subsection{Function spaces}

Let $\Phi_0$ be a Schwartz function so that $Supp(\widehat{\Phi_0})\subset \big\{\xi\in\mathbb{R}^d: |\xi|\leq 2 \big\}$ and $\widehat{\Phi_0}(\xi)=1$ for $|\xi|\leq 1$ and define $\phi:=\Phi_0-2^{-d}\Phi_0(2^{-1}\cdot)$ and $\phi_k:=2^{kd}\phi(2^k\cdot)$.
Then  $\{\phi_k\}_{k\in\mathbb{Z}}$ forms a (homogeneous) Littlewood-Paley partition of unity. That is,
 $Supp(\widehat{\phi_k})\subset \big\{\xi\in\mathbb{R}^d: 2^{k-1}\leq |\xi|\leq 2^{k+1} \big\}$ and
$\sum_{k\in\mathbb{Z}}{\widehat{\phi_k}(\xi)}=1$ for $\xi\not= 0$.

For $0<p,q\leq \infty$ and $\alpha\in\mathbb{R}$, the (homogeneous) Triebel-Lizorkin space $\dot{F}_p^{\alpha,q}(\rd)$ is defined by the collection of all $f\in S'/\mathcal{P}$ (tempered distribution modulo polynomials) such that 
 \begin{equation*}
 \Vert f\Vert_{\dot{F}_p^{\alpha,q}(\rd)}:=\big\Vert \big\{ 2^{\alpha k}\phi_k\ast f\big\}_{k\in\mathbb{Z}}\big\Vert_{L^p(\ell^q)}<\infty, \quad 0<p<\infty ~\text{ or }~ p=q=\infty,
 \end{equation*}
 \begin{equation*}
\Vert f\Vert_{\dot{F}_{\infty}^{\alpha,q}(\rd)}:=\sup_{P\in\mathcal{D}}{\Big( \frac{1}{|P|}\int_P{\sum_{k=-\log_2{\ell(P)}}^{\infty}{2^{\alpha kq}\big| \phi_k\ast f(x)\big|^q}}dx\Big)^{1/q}}, \quad 0<q<\infty
\end{equation*} 
where the supremum is taken over all dyadic cubes in $\mathbb{R}^d$.
 Then these spaces provide a general framework that unifies classical function spaces:
\begin{align*}
&\text{Hardy space} &\dot{{F}}_p^{0,2}(\rd)={H}^p(\rd)  & &0<p<\infty\\
&\text{Hardy-Sobolev space}&\dot{{F}}_p^{{\alpha},2}(\rd)={H}^p_{{\alpha}}(\rd)  & &  0<p<\infty\\
&BMO &\dot{{F}}_{{\infty}}^{0,2}(\rd)=BMO(\rd)\\
&\text{Sobolev-$BMO$} &\dot{{F}}_{{\infty}}^{\alpha,2}(\rd)=BMO_{\alpha}(\rd).
\end{align*}
Note that  $H^p(\rd)=L^p(\rd)$ if $1<p<\infty$.

\subsection{Maximal inequalities}

A crucial tool in theory of function spaces is the maximal inequalities of Fefferman and Stein \cite{Fe_St} and Peetre \cite{Pe}.

 Let $\mathcal{M}$  be the Hardy-Littlewood maximal operator, defined by
 \begin{equation*}
 \mathcal{M}f(x):=\sup_{Q: x\in Q}{\frac{1}{|Q|}\int_Q{|f(y)|}dy}
 \end{equation*} where the supremum is taken over all cubes containing $x$, and for $0<t<\infty$ let $\mathcal{M}_tf:=\big(  \mathcal{M}(|f|^t) \big)^{1/t}$.
Then the Fefferman-Stein vector-valued maximal inequality \cite{Fe_St} states that for $0<r<p,q<\infty$,
\begin{equation}\label{hlmax}
\Big\Vert  \Big(\sum_{k}{(\mathcal{M}_{r}f_k)^q}\Big)^{1/{q}} \Big\Vert_{L^p(\rd)} \lesssim  \Big\Vert \Big( \sum_{k}{|f_k|^q}  \Big)^{1/{q}}  \Big\Vert_{L^p(\rd)}.
\end{equation}  
The inequality (\ref{hlmax}) also holds for $0<p\leq \infty$ and $q=\infty$.

For $k\in\mathbb{Z}$ and $\sigma>0$ we now define the Peetre maximal operator $\mathfrak{M}_{\sigma,2^k}$ by the formula
\begin{equation*}
\mathfrak{M}_{\sigma,2^k}f(x):=\sup_{y\in\mathbb{R}^d}{\frac{|f(x-y)|}{(1+2^k|y|)^{\sigma}}}.
\end{equation*}
It is known in \cite{Pe} that for $f\in \mathcal{E}(A2^k)$,
\begin{equation}\label{maximalbound}
\mathfrak{M}_{d/r,2^k}f(x)\lesssim_A \mathcal{M}_rf(x) \qquad \text{ uniformly in }~ k.
\end{equation}   Then   (\ref{hlmax}) and (\ref{maximalbound}) yield the following maximal inequality: 
Suppose $f_k\in\mathcal{E}(A2^k)$ for some $A>0$. Then for $0<p<\infty$ or $p=q=\infty$, we have
\begin{equation}\label{maximal1}
\big\Vert \big\{\mathfrak{M}_{\sigma,2^k}f_k\big\}_{k\in\mathbb{Z}}\big\Vert_{L^p(\ell^q)}\lesssim_A \big\Vert \{ f_k\}_{k\in\mathbb{Z}}\big\Vert_{L^p(\ell^q)}.
\end{equation} if $\sigma>d/\min{(p,q)}$. 

Furthermore, a $\dot{F}_{\infty}$-version of (\ref{maximal1}) is recently given by the author \cite{Park1} :
Suppose $f_k\in\mathcal{E}(A2^k)$ for some $A>0$. Then for $0<q<\infty$ and $\mu\in\mathbb{Z}$, we have
\begin{equation}\label{maximal2}
\sup_{P\in\mathcal{D}_{\mu}}{\Big( \frac{1}{|P|}\int_P{\sum_{k=-\log_2{\ell(P)}}^{\infty}{\big(\mathfrak{M}_{\sigma,2^k}f_k(x) \big)^q}}dx\Big)^{1/q}}\lesssim\sup_{P\in\mathcal{D}_{\mu}}{\Big( \frac{1}{|P|}\int_P{\sum_{k=-\log_2{\ell(P)}}^{\infty}{\big|f_k(x) \big|^q}}dx\Big)^{1/q}}
\end{equation} uniformly in $\mu$ if $\sigma>d/q$.
We remark that (\ref{maximal2}) does not hold when $\mathfrak{M}_{\sigma,2^k}f_k$ is replaced by $\mathcal{M}_rf_k$ for all $0<r<\infty$.

As an application of (\ref{maximal2}),  we have
\begin{align}\label{embeddinginfty}
\big\Vert \big\{f_k\big\}_{k\geq \mu}\big\Vert_{L^{\infty}(\ell^{\infty})}&\lesssim \sup_{P\in\mathcal{D}: \ell(P)\leq 2^{-\mu}}{\Big( \frac{1}{|P|}\int_P{\sum_{k=-\log_2{\ell(P)}}^{\infty}{|f_k(x) |^{q}}}dx\Big)^{1/q}}.
\end{align}
See \cite{Park1} for more details.

\subsection{$\varphi$-transform in $\dot{F}_p^{0,q}$}

For a sequence of complex numbers $\bb:=\{b_Q\}_{Q\in\mathcal{D}}$ we define 
\begin{equation*}
\Vert \bb \Vert_{\dot{f}_p^{0,q}}:=\big\Vert  g^{q}(\bb)  \big\Vert_{L^p(\rd)}, \qquad 0<p<\infty \quad \text{or}\quad p=q=\infty
\end{equation*} 
\begin{equation*}
\Vert \bb\Vert_{\dot{f}_{\infty}^{0,q}}:=\sup_{P\in\mathcal{D}}{\Big(\frac{1}{|P|}\int_P{\sum_{Q\in \mathcal{D}, Q\subset P}{\big( |b_Q||Q|^{-1/2}\chi_Q(x)\big)^q}}dx \Big)^{1/q}}, \quad 0<q<\infty
\end{equation*}
where 
\begin{equation*}
g^{q}(\bb)(x):=\big\Vert \big\{|b_Q||Q|^{-1/2}\chi_Q(x) \big\}_{Q\in\mathcal{D}}\big\Vert_{\ell^q}.
\end{equation*}
Then the Triebel-Lizorkin space $\dot{F}_p^{0,q}(\rd)$ can be characterized by the discrete function space $\dot{f}_p^{0,q}$: For $Q\in\mathcal{D}$ let $x_Q$ be the lower left corner of $Q$.
Every $f\in \dot{F}_p^{0,q}(\rd)$ can be written as
\begin{equation*}
 f=\sum_{{Q\in\mathcal{D}}}{b_Q\varphi^Q} \qquad \text{ in }~ S'/\mathcal{P}
\end{equation*} 
where $\varphi_k$ and $\widetilde{\varphi_k}$ are Schwartz functions with localized frequency, involving Littlewood-Paley decomposition,  $\varphi^Q(x):=|Q|^{1/2}\varphi_k(x-x_Q)$, $\widetilde{\varphi}^Q(x):=|Q|^{1/2}\widetilde{\varphi_k}(x-x_Q)$  for each $Q\in\mathcal{D}_k$, and $b_Q:=\langle f,\widetilde{\varphi}^Q\rangle$.
To be specific, since $\sum_{k\in\mathbb{Z}}\widehat{\varphi_k}(\xi)\widehat{\widetilde{\varphi_k}}(\xi)=1$ for $\xi\not= 0$,
we have $f=\sum_{k\in\mathbb{Z}}\varphi_k\ast \widetilde{\varphi_k}\ast f$ in $S'/\mathcal{P}$ and for each $k\in \mathbb{Z}$
\begin{equation}\label{decomposition1}
\varphi_k\ast \widetilde{\varphi_k}\ast f(x)=\sum_{Q\in\mathcal{D}_k}b_Q\varphi^Q(x).
\end{equation}
Moreover, in the case, we have
\begin{equation}\label{decomposition2}
\Vert \bb \Vert_{\dot{f}_p^{0,q}} \lesssim \Vert  f  \Vert_{\dot{F}_p^{0,q}(\rd)}.
\end{equation}
The converse estimate is also true.  For any sequence $\bb:=\{b_Q\}_{Q\in\mathcal{D}}$ of complex numbers satisfying $\Vert  \bb\Vert_{\dot{f}_p^{0,q}}<\infty$,  
\begin{equation*}
f(x):=\sum_{Q\in\mathcal{D}}{b_Q\varphi^Q(x)}
\end{equation*} belongs to $\dot{F}_p^{0,q}$ and indeed,
\begin{equation}\label{converse}
\Vert  f  \Vert_{\dot{F}_p^{0,q}(\rd)} \lesssim \Vert  \bb \Vert_{\dot{f}_p^{0,q}}.
\end{equation}
See \cite{Fr_Ja0, Fr_Ja1} for more details.

\subsection{ Atomic decomposition of $\dot{f}_p^{0,q}$  }

Let $0<p\leq 1$ and $p\leq q\leq \infty$. A sequence of complex numbers $\mathbf{r}:=\{r_Q\}_{Q\in\mathcal{D}}$ is called an $\infty$-atom for $\dot{f}_p^{0,q}$ if there exists  $Q_0\in \mathcal{D}$ such that 
\begin{eqnarray*}
r_Q=0 \quad \text{if}\quad Q \not\subset Q_0
\end{eqnarray*}
 and \begin{eqnarray}\label{infdef}
\big\Vert  g^{q}(\mathbf{r})  \big\Vert_{L^{\infty}(\rd)}\leq |Q_0|^{-{1}/{p}}.
\end{eqnarray}
Then the following atomic decomposition of $\dot{f}_p^{0,q}$ holds:
\begin{lemma}\label{decomhardy}\cite{Fr_Ja2, Fr_Ja3}
Suppose $0<p\leq 1$, $p\leq q\leq\infty$, and $\bb:=\{b_Q\}_{Q\in\mathcal{D}}\in \dot{f}_p^{0,q}$. Then there exist $C_{p,q}>0$, a sequence of scalars $\{\lambda_j\}$, and a sequence of $\infty$-atoms $\mathbf{r}_j=\{r_{j,Q}\}_{{Q\in\mathcal{D}}}$ for $\dot{f}_p^{0,q}$ so that 
\begin{equation*}
\bb=\{b_Q\}_{Q\in\mathcal{D}}=\sum_{j=1}^{\infty}{\lambda_j\{r_{j,Q}\}_{Q\in\mathcal{D}}}=\sum_{j=1}^{\infty}{\lambda_j \mathbf{r}_j},
\end{equation*} and 
\begin{equation*}
\Big(\sum_{j=1}^{\infty}{|\lambda_j|^p}\Big)^{{1}/{p}}\leq C_{p,q}\big\Vert   \bb\big\Vert_{\dot{f}_{p}^{0,q}}.
\end{equation*}
Moreoever, it follows that 
\begin{eqnarray*}
\big\Vert  \bb  \big\Vert_{\dot{f}_p^{0,q}}\approx \inf{\Big\{ \Big(\sum_{j=1}^{\infty}{|\lambda_j|^p}\Big)^{{1}/{p}}   :  \bb=\sum_{j=1}^{\infty}{\lambda_j \mathbf{r}_j} ,~ \mathbf{r}_j ~\text{is a sequence of $\infty$-atoms for $\dot{f}_p^{0,q}$}    \Big\}}.
\end{eqnarray*}

\end{lemma}

\section{Characterizations of $L^p_A(\ell^q)$}
As mentioned in Section \ref{introduction}, we assume $A=2^{-2}$.

\subsection{Characterization of $L^p_A(\ell^q)$ by using a method of $\varphi$-transform}\label{phitransform}

We will study properties of $\{f_k\}_{k\in\mathbb{Z}}\in L^p_A(\ell^q)$, which are analogous to (\ref{decomposition1}), (\ref{decomposition2}), and (\ref{converse}).

Suppose that  $\Psi_0\in S(\rd)$ satisfies
\begin{equation*}
Supp(\widehat{\Psi_0})\subset \big\{\xi:|\xi| \leq 1 \big\} \qquad \text{ and }\qquad \widehat{\Psi_0}(\xi)=1\quad\text{for }~ |\xi|\leq 1/2.
\end{equation*}
 For each $k\in\mathbb{Z}$ and $Q\in\mathcal{D}_k$ let 
 $\Psi_k:=2^{kd}\Psi_0(2^k\cdot)$ and 
\begin{equation*}
\Psi^Q(x):=|Q|^{1/2}\Psi_k(x-x_Q)
\end{equation*} 
where $x_Q$ denotes the lower left corner of the cube $Q$ as before.

\begin{lemma}\label{discrete1} 
Let $0<p< \infty$ or $p=q=\infty$. 
\begin{enumerate}
\item Assume $f_k\in\mathcal{E}(2^{k-2})$ for each $k\in\mathbb{Z}$. 
 Then there exists a sequence of complex numbers $\bb:=\{b_Q\}_{Q\in\mathcal{D}}$ such that
\begin{equation*}
f_k(x)=\sum_{Q\in\mathcal{D}_k}{b_Q\Psi^Q(x)} \qquad \text{ and } \qquad \Vert \bb\Vert_{\dot{f}_p^{0,q}}\lesssim \big\Vert \big\{ f_k\big\}_{k\in\mathbb{Z}} \big\Vert_{L^p(\ell^q)}.
\end{equation*} 
\item   For any sequence $\bb=\{b_Q\}_{Q\in\mathcal{D}}$ of complex numbers satisfying $\Vert  \bb \Vert_{\dot{f}_p^{0,q}}<\infty$,  
\begin{equation*}
f_k(x):=\sum_{Q\in\mathcal{D}_k}{b_Q\Psi^Q(x)}
\end{equation*}   satisfies
\begin{equation}\label{converse2}
\big\Vert  \big\{ f_k\big\}_{k\in\mathbb{Z}}  \big\Vert_{L^p(\ell^q)} \lesssim \Vert  \bb \Vert_{\dot{f}_p^{0,q}}.
\end{equation}

\end{enumerate}
\end{lemma}

For the case $p=\infty$ and $0<q<\infty$ we introduce 
\begin{equation*}
\Vert \bb\Vert_{\dot{f}_{\infty}^{0,q}(\mu)}:=\sup_{P\in\mathcal{D}: \ell(P)\leq 2^{-\mu}}{\Big( \frac{1}{|P|}\int_P{\sum_{Q\in\mathcal{D}, Q\subset P}{\big(|b_Q||Q|^{-1/2}\chi_Q(x) \big)^q}}dx\Big)^{1/q}}
\end{equation*} for $\mu\in\zz$.

\begin{lemma}\label{discrete2}
Let $0<q<\infty$ and $\mu\in\mathbb{Z}$. 
\begin{enumerate}
\item Assume $f_k\in\mathcal{E}(2^{k-2})$ for each $k\geq \mu$. 
 Then there exists a sequence of complex numbers $\bb:=\{b_Q\}_{Q\in\mathcal{D},\ell(Q)\leq 2^{-\mu}}$ such that
\begin{equation*}
f_k(x)=\sum_{Q\in\mathcal{D}_k}{b_Q\Psi^Q(x)}
\end{equation*} and
\begin{equation*}
\Vert \bb\Vert_{\dot{f}_{\infty}^{0,q}(\mu)}\lesssim \sup_{P\in\mathcal{D}: \ell(P)\leq 2^{-\mu}}{\Big( \frac{1}{|P|}\int_P{\sum_{k=-\log_2{\ell(P)}}^{\infty}{|f_k(x)|^q}}dx\Big)^{1/q}}.
\end{equation*}
\item   For any sequence $\bb:=\{b_Q\}_{Q\in\mathcal{D},\ell(Q)\leq 2^{-\mu}}$ of complex numbers satisfying $\Vert  \bb \Vert_{\dot{f}_{\infty}^{0,q}(\mu)}<\infty$,  
\begin{equation*}
f_k(x):=\sum_{Q\in\mathcal{D}_k}{b_Q\Psi^Q(x)}
\end{equation*}   satisfies
\begin{equation*}
\sup_{P\in\mathcal{D}: \ell(P)\leq 2^{-\mu}}{\Big( \frac{1}{|P|}\int_P{\sum_{k=-\log_2{\ell(P)}}^{\infty}{|f_k(x)|^q}}dx\Big)^{1/q}}  \lesssim \Vert  \bb \Vert_{\dot{f}_{\infty}^{0,q}(\mu)}.
\end{equation*}

\end{enumerate}

\end{lemma}     
\begin{proof}[Proof of Lemma \ref{discrete1}]
(1) Since $Supp(\widehat{f_k}(2^k\cdot))\subset \{|\xi|\leq 1/2\}$, $\widehat{f_k}$ admits the decomposition
\begin{equation*}
\widehat{f_k}(\xi)=2^{-kd}\sum_{l\in\mathbb{Z}^d}{f_k(2^{-k}l)e^{-2\pi i\langle 2^{-k}l,\xi\rangle}},
\end{equation*} using a scaling argument and  the Fourier series representation of $\widehat{f_k}(2^k\cdot)$.
Then we have
\begin{align}
f_k(x)&=\big(\widehat{f_k}\widehat{\Psi_k} \big)^{\vee}(x)=2^{-kd}\sum_{l\in\mathbb{Z}^d}{f_k(2^{-k}l)\Psi_k(x-2^{-k}l)}\nonumber\\
 &=\sum_{l\in\mathbb{Z}^d}{2^{-kd/2}f_k(2^{-k}l)2^{-kd/2}\Psi_k(x-2^{-k}l)}. \label{decompositionfourier}
\end{align}
For any $Q\in\mathcal{D}_k$ we write
\begin{equation*}
Q=Q_{k,l}:=\{x\in\mathbb{R}^d:2^{-k}l_i\leq x_i\leq 2^{-k}(l_i+1), ~i=1,\dots,d\}
\end{equation*}
where $l:=(l_1,\dots,l_d)\in\mathbb{Z}^d$. That is, $Q_{k,l}$ is the dyadic cube, contained in $\mathcal{D}_{k}$, whose lower left corner is $x_{Q_{k,l}}=2^{-k}l$.
Now we use the notations
\begin{equation*}
b_{Q_{k,l}}:=2^{-kd/2}f_k(2^{-k}l)=|Q_{k,l}|^{1/2}f_k(x_{Q_{k,l}}),
\end{equation*}
\begin{equation*}
\Psi^{Q_{k,l}}(x):=2^{-kd/2}\Psi_k(x-2^{-k}l)=|Q_{k,l}|^{1/2}\Psi_k(x-x_{Q_{k,l}}).
\end{equation*} 
Then (\ref{decompositionfourier}) can be expressed as
\begin{equation}\label{expressfk}
f_k(x)=\sum_{Q\in\mathcal{D}_k}{b_Q\Psi^Q(x)}.
\end{equation}
In addition, for a.e. $x\in\mathbb{R}^d$ there exists the unique dyadic cube $Q_0\in\mathcal{D}_k$ whose interior contains $x$, and this yields that
\begin{equation}\label{shareidea}
\sum_{Q\in\mathcal{D}_k}{|b_Q||Q|^{-1/2}\chi_Q(x)}=|b_{Q_0}||Q_0|^{-1/2}=|f_k(x_{Q_0})|\lesssim \mathfrak{M}_{\sigma,2^k}f_k(x) \quad \text{a.e.}~x.
\end{equation} Here, the inequality holds due to the fact that
\begin{equation}\label{infmax}
\sup_{y\in Q}{|f_k(y)|}\lesssim \inf_{y\in Q}{\mathfrak{M}_{\sigma,2^k}f_k(y)} \quad \text{ uniformly in }~ Q\in\mathcal{D}_k,
\end{equation} which is valid even for $f_k$ without Fourier support condition. 
 Then we can easily see that for $\sigma>d/\min{(p,q)}$, using (\ref{shareidea}) and (\ref{maximal1}),
\begin{align*}
\Vert \bb \Vert_{\dot{f}_p^{0,q}} &=\big\Vert \big\{|b_Q| |Q|^{-1/2}\chi_Q\big\}_{Q\in\mathcal{D}}\big\Vert_{L^p(\ell^q)}=\Big\Vert \Big\{ \sum_{Q\in\mathcal{D}_k}{|b_Q||Q|^{-1/2}\chi_Q}\Big\}_{k\in\mathbb{Z}}\Big\Vert_{L^p(\ell^q)}\\
&\lesssim \big\Vert \big\{\mathfrak{M}_{\sigma,2^k}f_k\big\}_{k\in\mathbb{Z}}\big\Vert_{L^p(\ell^q)}\lesssim \big\Vert \big\{ f_k\big\}_{k\in\mathbb{Z}} \big\Vert_{L^p(\ell^q)},
\end{align*} as desired.

(2) For a given $\bb:=\{b_Q\}_{Q\in\mathcal{D}}$ and $k\in\mathbb{Z}$ let
\begin{equation*}
f_k(x):=\sum_{Q\in\mathcal{D}_k}{b_Q\Psi^Q(x)}.
\end{equation*} 
Setting
\begin{equation*}
E^k_0(x):=\big\{Q\in\mathcal{D}_k:|x-x_Q|<2^{-k} \big\}
\end{equation*}
\begin{equation*}
E^k_j(x):=\big\{Q\in\mathcal{D}_k:2^{-k+j-1}\leq |x-x_Q|<2^{-k+j} \big\},\quad j\in\mathbb{N}
\end{equation*}
for each $k\in\mathbb{Z}$ and $x\in\mathbb{R}^d$, we can write
\begin{equation*}
|f_k(x)|\leq \sum_{j=0}^{\infty}{\sum_{Q\in E_j^k(x)}{|b_Q|\big| \Psi^Q(x)\big|}}.
\end{equation*}
Choose $0<\epsilon<\min{(1,p,q)}$ and $M>d/\epsilon$.
Observe that $|\Psi^Q(x)|\lesssim_M 2^{-jM}|Q|^{-1/2}$ on $E_j^k$ and then the embedding $\ell^{\epsilon}\hookrightarrow \ell^1$ shows that
\begin{align*}
|f_k(x)|&\lesssim \sum_{j=0}^{\infty}{2^{-jM}\Big(\sum_{Q\in E_j^k(x)}{\big(|b_Q||Q|^{-1/2}\big)^{\epsilon}}\Big)^{1/\epsilon}}\\
&\approx \sum_{j=0}^{\infty}{2^{-j(M-d/\epsilon)}\Big(\frac{1}{2^{-kd}2^{jd}}\int_{\mathbb{R}^d}{\sum_{Q\in E_j^k(x)}{\big(|b_Q||Q|^{-1/2}\chi_Q(y)\big)^{\epsilon}}}dy \Big)^{1/\epsilon}}\\
&\lesssim  \mathcal{M}_{\epsilon}\Big( \sum_{Q\in\mathcal{D}_k}{|b_Q||Q|^{-1/2}\chi_Q}\Big)(x).
\end{align*} 
Finally, as a result of the maximal inequality (\ref{hlmax}), we obtain
\begin{equation*}
\big\Vert \big\{ f_k\big\}_{k\in\mathbb{Z}} \big\Vert_{L^p(\ell^q)} \lesssim \Big\Vert \Big\{ \sum_{Q\in\mathcal{D}_k}{|b_Q||Q|^{-1/2}\chi_Q}\Big\}_{k\in\mathbb{Z}}\Big\Vert_{L^p(\ell^q)}=\Vert \bb\Vert_{\dot{f}_p^{0,q}},
\end{equation*}
as required.   \qedhere

\end{proof}

\begin{proof}[Proof of Lemma \ref{discrete2}]
(1) The proof is very similar to that of Lemma \ref{discrete1}. Indeed, using (\ref{expressfk}), (\ref{shareidea}) and (\ref{maximal2}) with $\sigma>d/q$, it can be verified that
\begin{align*}
\Vert \bb\Vert_{\dot{f}_{\infty}^{0,q}(\mu)}&=\sup_{P\in\mathcal{D}: \ell(P)\leq 2^{-\mu}}{\Big( \frac{1}{|P|}\int_P{\sum_{k=-\log_2{\ell(P)}}^{\infty}{\Big(\sum_{Q\in \mathcal{D}_k}{|b_Q||Q|^{-1/2}\chi_Q(x)} \Big)^q}}dx\Big)^{1/q}}\\
&\lesssim \sup_{P\in\mathcal{D}: \ell(P)\leq 2^{-\mu}}{\Big(\frac{1}{|P|}\int_P{\sum_{k=-\log_2{\ell(P)}}^{\infty}{\big(\mathfrak{M}_{\sigma,2^k}f_k(x) \big)^q}}dx \Big)^{1/q}}\\
&\lesssim \sup_{P\in\mathcal{D}: \ell(P)\leq 2^{-\mu}}{\Big( \frac{1}{|P|}\int_P{\sum_{k=-\log_2{\ell(P)}}^{\infty}{|f_k(x)|^q}}dx\Big)^{1/q}}.
\end{align*}

(2) 
We note that
\begin{equation}\label{finftymu}
\Vert \bb\Vert_{\dot{f}_{\infty}^{0,q}(\mu)}=\sup_{P\in\mathcal{D}: \ell(P)\leq 2^{-\mu}}{\Big( \frac{1}{|P|}\sum_{Q\in\mathcal{D},Q\subset P}{\big( |b_Q||Q|^{-1/2}\big)^q|Q|}\Big)^{1/q}}.
\end{equation}
Let
\begin{equation*}
f_k(x):=\sum_{Q\in\mathcal{D}_k}{b_Q\Psi^Q(x)}
\end{equation*} and choose $M>d/\min{(1,q)}$. 
Using H\"older's inequality if $q>1$ or the embedding $\ell^q\hookrightarrow \ell^1$ if $q\leq 1$, we obtain
\begin{align*}
|f_k(x)|&\lesssim_M \sum_{Q\in\mathcal{D}_k}{|b_Q||Q|^{-1/2}\frac{1}{(1+2^k|x-x_Q|)^{2M}}}\\
&\lesssim \Big( \sum_{Q\in\mathcal{D}_k}{\big( |b_Q||Q|^{-1/2}\big)^q\frac{1}{(1+2^k|x-x_Q|)^{Mq}}}\Big)^{1/q},
\end{align*}
which further implies that
\begin{align*}
&\sup_{P\in\mathcal{D}: \ell(P)\leq 2^{-\mu}}{\Big( \frac{1}{|P|}\int_P{\sum_{k=-\log_2{\ell(P)}}^{\infty}{|f_k(x)|^q}}dx\Big)^{1/q}}\\
&\lesssim \sup_{P\in\mathcal{D}: \ell(P)\leq 2^{-\mu}}{\Big({\sum_{k=-\log_2{\ell(P)}}^{\infty}{\sum_{Q\in\mathcal{D}_k}{\big(|b_Q||Q|^{-1/2} \big)^q \frac{1}{|P|}\int_P\frac{1}{(1+2^k|x-x_Q|)^{Mq}} dx  }}} \Big)^{1/q}}.
\end{align*}
For each $P\in\mathcal{D}$ and $m\in\mathbb{Z}^d$ let $P+\ell(P)m:=\big\{ x+\ell(P)m:x\in P\big\}$ and denote by $\mathcal{D}_k(P,m)$ the subfamily of $\mathcal{D}_k$ that contains any dyadic cubes belonging to $P+\ell(P)m$.
Then in the last expression we decompose 
\begin{equation*}
\sum_{Q\in\mathcal{D}_k}=\sum_{m\in\mathbb{Z}^d, |m|\leq 2d}\sum_{Q\in\mathcal{D}_k(P,m)}+\sum_{m\in\mathbb{Z}^d, |m|> 2d}\sum_{Q\in\mathcal{D}_k(P,m)}=:\mathcal{I}_{k,M}^{P}+\mathcal{J}_{k,M}^{P}
\end{equation*} which is possible because $P$ and $Q$'s are dyadic cubes with $\ell(Q)=2^{-k}\leq \ell(P)$.

We first see that
\begin{align*}
\Big(\sum_{k=-\log_2{\ell(P)}}^{\infty}{\mathcal{I}_{k,M}^{P}} \Big)^{1/q}&\lesssim  \sum_{m\in\mathbb{Z}^d,|m|\leq 2d}{\Big( \frac{1}{|P|}\sum_{k=-\log_2{\ell(P)}}^{\infty}{\sum_{Q\in\mathcal{D}_k(P,m)}{\big(|b_Q||Q|^{-1/2}\big)^q|Q|}}\Big)^{1/q}}\\
&\lesssim \sup_{R\in\mathcal{D}: \ell(R)=\ell(P)}{\Big( \frac{1}{|R|}\sum_{k=-\log_2{\ell(R)}}^{\infty}{\sum_{Q\in\mathcal{D}_k, Q\subset R}{\big(|b_Q||Q|^{-1/2}\big)^q|Q|}}\Big)^{1/q}}.
\end{align*}

On the other hand, if $|m|>2d$ and $Q\in\mathcal{D}_k(P,m)$ then
\begin{equation*}
|x-x_Q|\gtrsim \ell(P)|m|,
\end{equation*} and therefore
\begin{equation*}
\mathcal{J}_{k,M}^{P}\lesssim \sum_{m\in\mathbb{Z}^d, |m|> 2d}{\frac{1}{|m|^{Mq}}\frac{1}{2^{kMq}}\frac{1}{\ell(P)^{Mq}}\sum_{Q\in\mathcal{D}_k(P,m)}{\big(|b_Q||Q|^{-1/2} \big)^q}}.
\end{equation*}
Now we apply the triangle inequality if $q\geq 1$ or $\ell^q\hookrightarrow \ell^1$ if $q<1$ to obtain that
\begin{align*}
&\Big(\sum_{k=-\log_2{\ell(P)}}^{\infty}{\mathcal{J}_{k,M}^{P}} \Big)^{\min{(1,q)}/q}\\
&\lesssim  \sum_{\substack{m\in\mathbb{Z}^d\\ |m|>2d}}\frac{1}{|m|^{M\min{(1,q)}}}\Big( \sum_{k=-\log_2{\ell(P)}}^{\infty}{\frac{1}{2^{kMq}}\frac{1}{\ell(P)^{Mq}}\sum_{\substack{Q\in\mathcal{D}_k\\ Q\subset P+m\ell(P)}}{\big(|b_Q||Q|^{-1/2} \big)^q   }}\Big)^{\min{(1,q)}/q}.
\end{align*}
Since $M\min{(1,q)}>d$ and $2^k\ell(P)\geq 1$, the above expression is bounded by
\begin{align*}
&\sum_{\substack{m\in\mathbb{Z}^d\\ |m|>2d}}{\frac{1}{|m|^{M\min{(1,q)}}}\Big( \frac{1}{|P|}\sum_{k=-\log_2{\ell(P)}}^{\infty}{\sum_{\substack{Q\in\mathcal{D}_k\\ Q\subset P+m\ell(P)}}{\big( |b_Q||Q|^{-1/2}\big)^q|Q|}}\Big)^{\min{(1,q)}/q}}\\
&\lesssim \sup_{R\in\mathcal{D}: \ell(R)=\ell(P)}{\Big(\frac{1}{|R|}\sum_{k=-\log_2{\ell(R)}}^{\infty}{\sum_{Q\in\mathcal{D}_k, Q\subset R}{\big( |b_Q||Q|^{-1/2}\big)^q|Q|}} \Big)^{\min{(1,q)}/q}}.
\end{align*}

Combining these estimates, taking a supremum over ${P\in\mathcal{D}, \ell(P)\leq 2^{-\mu}}$, and using (\ref{finftymu}), we conclude that
\begin{align*}
&\sup_{P\in\mathcal{D}: \ell(P)\leq 2^{-\mu}}{\Big( \frac{1}{|P|}\int_P{\sum_{k=-\log_2{\ell(P)}}^{\infty}{|f_k(x)|^q}}dx\Big)^{1/q}} \\
&\lesssim  \sup_{R\in\mathcal{D}: \ell(R)\leq 2^{-\mu}}{\Big( \frac{1}{|R|}\sum_{k=-\log_2{\ell(R)}}^{\infty}{\sum_{Q\in\mathcal{D}_k, Q\subset R}{\big(|b_Q||Q|^{-1/2}\big)^q|Q|}}\Big)^{1/q}}\leq \Vert \bb\Vert_{\dot{f}_{\infty}^{0,q}(\mu)}.    \qedhere
\end{align*}

\end{proof}

\subsection{Characterization of $L^p_A(\ell^q)$ by using a sharp maximal function}

Given a locally integrable function $f$ on $\mathbb{R}^d$ the Fefferman-Stein sharp maximal function $f^{\sharp}$ is defined by 
\begin{equation*}
f^{\sharp}(x):=\sup_{P:x\in P}\frac{1}{|P|}\int_P{|f(y)-f_P|}dy
\end{equation*} where $f_P:=\frac{1}{|P|}\int_P{f(z)}dz$ and the supremum is taken over all cubes $P$ containing $x$ ( not necessarily dyadic cubes ). Then a fundamental inequality of Fefferman and Stein \cite{Fe_St1} says that 
for $1<p<\infty$ and $1\leq p_0\leq p$, if $f\in L^{p_0}(\mathbb{R}^d)$,  then we have 
\begin{equation}\label{fsmaximal}
\Vert \mathcal{M}f\Vert_{L^p(\rd)}\lesssim \Vert f^{\sharp}\Vert_{L^p(\rd)}.
\end{equation}
Using this result, it can be proved that for $0<q<p<\infty$,
\begin{equation}\label{originalsharpcharacter}
\Vert f\Vert_{\dot{F}_p^{0,q}(\rd)}\approx \Big\Vert \sup_{P:x\in P}\Big( \frac{1}{|P|}\int_P{\sum_{k\geq -\log_2{\ell(P)}}{\big| \phi_k\ast f(y)\big|^q}}dy\Big)^{1/q} \Big\Vert_{L^{p}(x)}
\end{equation} where the supremum in the $L^p$-norm is taken over all cubes containing $x$.
See \cite{Park4}, \cite[Proposition 6.1 and 6.2]{Se} for more details.

By following the proof of the estimate (\ref{fsmaximal}) in \cite{Fe_St1} we can actually replace the maximal functions by dyadic maximal ones.
For locally integrable function $f$ we define
 the dyadic maximal function 
\begin{equation*}
\mathcal{M}^{(d)}f(x):=\sup_{P\in \mathcal{D}: x\in P}{\frac{1}{|P|}\int_P{|f(y)|}dy},
\end{equation*} 
and the dyadic sharp maximal funtion
\begin{equation*}
\mathcal{M}^{\sharp}f(x):= \sup_{P\in \mathcal{D}:x\in P }{\frac{1}{|P|}\int_P{|f(y)-f_P|}dy}
\end{equation*} where the supremums are taken over all dyadic cubes $P$ containing $x$.
Then for $1<p<\infty$, $1\leq p_0\leq p$, and $f\in L^{p_0}$ we have
\begin{equation}\label{dyadicmaximal}
\Vert \mathcal{M}^{(d)}f\Vert_{L^p(\rd)}\lesssim_p \Vert \mathcal{M}^{\sharp}f\Vert_{L^p(\rd)}.
\end{equation}

We now provide a characterization of $L^p_{A}(\ell^q)$ for $0<q<p<\infty$, which is the analogue of (\ref{originalsharpcharacter}).
\begin{lemma}\label{charactersharp}
Let $0<q<p<\infty$. Suppose $f_k\in\mathcal{E}(2^{k-2})$ for each $k\in\mathbb{Z}$. Then
\begin{equation}\label{sharpequiv}
\big\Vert \big\{ f_k\big\}_{k\in\mathbb{Z}}\big \Vert_{L^p(\ell^q)}\approx \Big\Vert    \sup_{P\in\mathcal{D}: x\in P}\Big( \frac{1}{|P|}\int_P{\sum_{k=-\log_2{\ell(P)}}^{\infty}{|f_k(y)|^q}}dy\Big)^{1/q}           \Big\Vert_{L^p(x)}
\end{equation}
where the supremum is taken over all dyadic cubes containing $x$.
\end{lemma}
The proof of the above lemma is almost same as that of \cite[Lemma 2.3]{Park4}, and for completeness we give a brief proof here.
\begin{proof}
 The direction $"\gtrsim"$ is immediate because the right-hand side of (\ref{sharpequiv}) is bounded by 
$\big\Vert \mathcal{M}_q\big( \big\Vert \{f_k\}_{k\in\zz}\big\Vert_{\ell^q}\big)\big\Vert_{L^p(\rd)}$ and the $L^p$-boundedness of $\mathcal{M}_q$ yields the desired estimate.

For the opposite direction, using (\ref{dyadicmaximal}),  the left-hand side of (\ref{sharpequiv}) is smaller than a constant times
\begin{equation*}
\Big\Vert \mathcal{M}^{\sharp}\Big(\sum_{k\in\zz}{|f_k|^q}\Big)\Big\Vert_{L^{p/q}(\rd)}^{1/q}
\end{equation*}
and the sharp maximal function can be controlled by the sum of
\begin{equation*}
\sup_{P\in\mathcal{D}:x\in P}{\frac{1}{|P|}\int_P{\sum_{k=-\log_2{\ell(P)}}^{\infty}{|f_k(y)|^q}}dy},
\end{equation*}
\begin{equation*}
\mathfrak{N}^q\big(\{f_k\}_{k\in\zz}\big):=\sup_{P\in\mathcal{D}:x\in P}{\frac{1}{|P|}\int_{P}\frac{1}{|P|}\int_P{\sum_{k=-\infty}^{-\log_2{\ell(P)}-1}{\big|f_k(y)-f_k(z) \big|^q}}dzdy}.
\end{equation*}
The first term clearly gives the expected upper bound and thus it is enough to show that 
\begin{equation}\label{npqsharp}
\mathfrak{N}^q\big(\{f_k\}_{k\in\zz}\big) \lesssim   \sup_{P\in\mathcal{D}: x\in P}\frac{1}{|P|}\int_P{\sum_{k=-\log_2{\ell(P)}}^{\infty}{|f_k(y)|^q}}dy.
\end{equation}
If $\ell(P)\leq 2^{-k-1}$ then there exists the unique dyadic cube $Q_P\in\mathcal{D}_{k}$ containing $P$. 
Then, using Taylor's formula, we can bound $ \mathfrak{N}^q\big(\{f_k\}_{k\in\zz}\big)$ by
\begin{equation*}
\sup_{P\in\mathcal{D}:x\in P}{ \sum_{k=-\infty}^{-\log_2{\ell(P)}-1}{\big( 2^kl(P)\big)^q\big( \sup_{w\in Q_P}{|\psi_k|\ast |f_k|(w)}\big)^q}}
\end{equation*}
for some $\psi_k\in S(\rd)$ with $Supp(\widehat{\psi_k})\subset \big\{\xi\in \rd:|\xi|\lesssim 2^k \big\}$.
Moreover, (\ref{infmax}) implies that that for any $\sigma>0$ 
\begin{align*}
\sup_{w\in Q_P}{|\psi_k|\ast |f_k|(w)}&\lesssim_{\sigma}\inf_{w\in Q_P}{\mathfrak{M}_{\sigma,2^k}\big(|\psi_k|\ast |f_k|\big)(w)}\\
 &\lesssim \inf_{w\in Q_P}{\mathfrak{M}_{\sigma,2^k}\big( \mathfrak{M}_{\sigma,2^k}f_k\big)(w)}\lesssim \inf_{w\in Q_P}{\mathfrak{M}_{\sigma,2^k}f_k(w)}
\end{align*}
and this yields that 
\begin{align*}
 \mathfrak{N}^q\big(\{f_k\}_{k\in\zz}\big)&\lesssim \sup_{P\in\mathcal{D} : x\in P} \sum_{k=-\infty}^{-\log_2{\ell(P)}-1}{\big( 2^k\ell(P)\big)^q\Big(\inf_{w\in Q_P}{\mathfrak{M}_{\sigma,2^k}f_k(w)}\Big)^q}\\
 &\lesssim \sup_{P\in\mathcal{D} : x\in P}{\sup_{k\in\zz}{\inf_{w\in Q_P}{\big(\mathfrak{M}_{\sigma,2^k}f_k(w) \big)^q}}}.
 \end{align*}
 We observe that for each $Q_P\in\mathcal{D}_{k}$, the infimum over $w\in O_P$ in the preceding expression is less than
 \begin{align*}
 \inf_{w\in Q_P}{\sum_{l=-\log_2{\ell(Q_P)}}^{\infty}{\big( \mathfrak{M}_{\sigma,2^l}f_l(w)\big)^q}}&\leq \frac{1}{|Q_P|}\int_{Q_P}{\sum_{l=-\log_2{\ell(Q_P)}}^{\infty}{\big( \mathfrak{M}_{\sigma,2^l}f_l(w)\big)^q}}dw\\
 &\leq \sup_{Q\in\mathcal{D} : x\in Q}{ \frac{1}{|Q|}\int_{Q}{\sum_{l=-\log_2{\ell(Q)}}^{\infty}{\big( \mathfrak{M}_{\sigma,2^l}f_l(w)\big)^q}}dw}
 \end{align*} since $x\in P\subset Q_P$.
 Choosing $\sigma>p,q$, the last expression can be further controlled by 
 \begin{equation*}
 \sup_{Q\in\mathcal{D} : x\in Q}{ \frac{1}{|Q|}\int_{Q}{\sum_{l=-\log_2{\ell(Q)}}^{\infty}{\big| f_l(w)\big|^q}}dw}.
 \end{equation*}
 The proof of this estimate is contained in \cite[Lemma 2.2]{Park4} and we omit it here. This completes the proof of (\ref{npqsharp}).

\end{proof}

\section{Dualization of $L_A^p(\ell^q)$ via a discrete function space $\dot{f}_{p}^{0,q}$}\label{dualization}

Suppose $1<p<\infty$ and $1\leq q<\infty$. Let $1<p'<\infty$ and $1<q'\leq \infty$ be the H\"older conjugates of $p$ and $q$, respectively.
Then it is known in \cite{Fr_Ja2} that the dual of $\dot{f}_p^{0,q}$ is $\dot{f}_{p'}^{0,q'}$.
Indeed, for $\{b_Q\}_{Q\in\mathcal{D}}\in \dot{f}_{p'}^{0,q'}$
\begin{equation}\label{discretedual}
\Vert \{b_Q\}_{Q\in\mathcal{D}} \Vert_{\dot{f}_{p'}^{0,q'}}=\sup_{\{r_Q\}_{Q\in\mathcal{D}}:\Vert \{r_Q\}_{Q\in\mathcal{D}}\Vert_{\dot{f}_{p}^{0,q}}\leq 1}{\Big| \sum_{Q\in\mathcal{D}}{b_Q r_Q}\Big|}.
\end{equation}

In this section, we dualize $L^{p}_A(\ell^q)$ through the relationship between the vector-valued space $L_A^p(\ell^q)$ and the discrete space $\dot{f}_p^{0,q}$ in Lemma \ref{discrete1}.

For any $\{f_k\}_{k\in\zz}\in L^p_A(\ell^q)$ and $Q\in\mathcal{D}$ we define the operator $\mathfrak{U}_Q$ by
\begin{equation*}
\mathfrak{U}_Q\big(\{f_k\}_{k\in\zz}\big):= |Q|^{1/2}f_{-\log_2{\ell(Q)}}(x_Q)
\end{equation*} where we recall that $x_Q$ is the lower left corner of $Q\in\mathcal{D}$.
Furthermore, for any $\{r_Q\}_{Q\in\mathcal{D}}\in \dot{f}_p^{0,q}$ and $k\in\zz$ we define the operator $\mathfrak{V}_k^{\Psi_0}$ by
\begin{equation*}
\mathfrak{V}_k^{\Psi_0}\big(\{r_Q\}_{Q\in\mathcal{D}}\big)(x):=\sum_{Q\in\mathcal{D}_k}{r_Q\Psi^Q(x)}.
\end{equation*}
Then for each $k\in\zz$
\begin{equation*}
\mathfrak{V}_k^{\Psi_0}\big(\big\{\mathfrak{U}_Q(\{f_j\}_{j\in\zz})\big\}_{Q\in\mathcal{D}}\big)(x)=f_k(x)
\end{equation*}
and it follows from Lemma \ref{discrete1} (2) that 
\begin{equation}\label{discretecha2}
\big\Vert \big\{\mathfrak{V}_k^{\Psi_0}\big(\{r_Q\}_{Q\in\mathcal{D}}\big)\big\}_{k\in \zz}\big\Vert_{L^p(\ell^q)}\lesssim \Vert \{r_Q\}_{Q\in\mathcal{D}}\Vert_{\dot{f}_p^{0,q}}.
\end{equation}

\begin{lemma}\label{dualitylemma}
Let $1<p<\infty$ and $1\leq q<\infty$.
Suppose $f_k\in \mathcal{E}(2^{k-2})$ for $k\in\zz$.
Then
\begin{equation*}
\Vert \{f_k\}_{k\in\zz}\Vert_{L^{p'}(\ell^{q'})}\lesssim\sup_{\{r_Q\}_{Q\in\mathcal{D}}:\Vert \{r_Q\}_{Q\in\mathcal{D}}\Vert_{\dot{f}_{p}^{0,q}}\leq 1}{\Big|\int_{\rd}{\sum_{k\in\zz}{f_k(x)\mathfrak{V}_k^{\Psi_0}\big(\{r_Q\}_{Q\in\mathcal{D}}\big)(x)}}dx \Big|}
\end{equation*}
\end{lemma}
\begin{proof}
By using Lemma \ref{discrete1} and (\ref{discretedual})
\begin{align*}
\Vert \{f_k\}_{k\in\zz}\Vert_{L^{p'}(\ell^{q'})}&=\big\Vert \big\{ \mathfrak{U}_Q\big(\{f_l\}_{l\in\zz}\big)\big\}_{Q\in\mathcal{D}} \big\Vert_{\dot{f}_{p'}^{0,q'}}\\
 &=\sup_{\{r_Q\}_{Q\in\mathcal{D}}:\Vert \{r_Q\}_{Q\in\mathcal{D}}\Vert_{\dot{f}_{p}^{0,q}}\leq 1}{\Big|\sum_{k\in\zz}\sum_{Q\in\mathcal{D}_k}{\mathfrak{U}_Q\big(\{f_l\}_{l\in\zz}\big) r_Q} \Big|}.
\end{align*}
Moreover, for each $k\in\zz$ 
\begin{align*}
\sum_{Q\in\mathcal{D}_k}\mathfrak{U}_Q\big(\{f_l\}_{l\in\zz}\big) r_Q&=\sum_{Q\in\mathcal{D}_k}2^{-kd/2}f_k(x_Q)r_Q=\sum_{Q\in\mathcal{D}_k}2^{-kd/2}\wt{\Psi_k}\ast f_k(x_Q)r_Q\\
 &=\int_{\rd}{f_k(x)\Big(\sum_{Q\in\mathcal{D}_k}r_Q2^{-kd/2} \Psi_k(x-x_Q) \Big)}dx\\
 &=\int_{\rd}{f_k(x)\mathfrak{V}_k^{\Psi_0}\big(\{r_Q\}_{Q\in\mathcal{D}}\big)(x)}dx
\end{align*} where $\wt{\Psi_k}:=\Psi_k(-\cdot)$, and this proves the lemma.
\end{proof}


\section{Complex Interpolation theorem for multipliers on $L^p_A(\ell^q)$}\label{complexsection}

In this section, we obtain an interpolation theorem for multipliers on $L^p_A(l^q)$ by using
 the complex method of Triebel \cite[2.4.9]{Tr}, which is a generalization of the well-known results of Calder\'on \cite{Ca} and Calder\'on and Torchinsky \cite{Ca_To}.

Let $\Omega:=\{z\in \cc : 0<Re(z)<1 \}$ be a strip in the complex plane $\cc$ and $\overline{\Omega}$ denote its closure.
We say that the mapping $z \mapsto f^z\in S'(\rn)$ is a $S'$-analytic function in $\Omega$ if the following properties are satisfied:
\begin{enumerate}
\item For any $\varphi\in S(\rn)$ with compact support,
$g(x,z):=\big(\varphi \widehat{f^z} \big)(x)$ is a uniformly continuous and bounded function in $\rn\times \overline{\Omega}$.
\item For any $\varphi\in S(\rn)$ with compact support and any fixed $x\in\rn$,
$h_x(z):=\big( \varphi \widehat{f^z}\big)^{\vee}(x)$ is an analytic function in $\Omega$.
\end{enumerate}

Let $0<p_0,p_1,q_0,q_1<\infty$.
Then we define
$F\big(L^{p_0}_A(\ell^{q_0}),L^{p_1}_A(\ell^{q_1})\big)$ to be the collection of all systems $\ff^z:= \{f_k^z\}_{k\in\zz}$ such that each $f_k^z$ is a $S'$-analytic function  in $\Omega$, 
\begin{equation*}
\ff^{it}=\{f_k^{it}\}_{k\in\zz}\in L_A^{p_0}(\ell^{q_0}), \qquad \ff^{1+it}=\{ f_k^{1+it}\}_{k\in\zz}\in L_A^{p_1}(\ell^{q_1}) \quad  \text{ for any }~ t\in \rr,
\end{equation*} and
\begin{equation*}
\sup_{t\in \rr}{\big\Vert \ff^{l+it}\big\Vert_{L^{p_l}(\ell^{q_l})}}<\infty \qquad \text{ for each }~ l=1,2.
\end{equation*}
Moreover, for $\ff^z \in F(L_A^{p_0}(\ell^{q_0}),L_A^{p_1}(\ell^{q_1}))$,
\begin{equation*}
\Vert \ff^z \Vert_{F(L_A^{p_0}(\ell^{q_0}),L_A^{p_1}(\ell^{q_1}))}:=\max{\Big( \sup_{t\in \rr}{\Vert \ff^{it}\Vert_{L^{p_0}(\ell^{q_0})}}, \sup_{t\in \rr}{\Vert\ff^{1+it}\Vert_{L^{p_1}(\ell^{q_1})}}\Big)}.
\end{equation*}
For $0<\theta<1$ the intermediate space $( L_A^{p_0}(\ell^{q_0}), L_A^{p_1}(\ell^{q_1}))_{\theta}$ is defined by
\begin{equation*}
\big( L_A^{p_0}(\ell^{q_0}), L_A^{p_1}(\ell^{q_1})\big)_{\theta}:=\big\{\{ f_k\}_{k\in\zz}: \exists \ff^z=\{f_k^z\}_{z\in\zz}\in F\big(L_A^{p_0}(\ell^{q_0}), L_A^{p_1}(\ell^{q_1})\big) \text{ s.t. } f_k=f_k^{\theta} \big\}
\end{equation*}
and the (quasi-)norm in the space is 
\begin{equation*}
\Vert \{f_k\}_{k\in\zz}\Vert_{( L_A^{p_0}(\ell^{q_0}), L_A^{p_1}(\ell^{q_1}))_{\theta}}:=\inf_{\ff^z\in F(L_A^{p_0}(\ell^{q_0}), L_A^{p_1}(\ell^{q_1})):f_k=f_k^{\theta}}{\Vert \ff^z\Vert_{F(L_A^{p_0}(\ell^{q_0}), L_A^{p_1}(\ell^{q_1}))}}
\end{equation*}
where the infimum is taken over all admissible system $\ff^z=\{f_k^z\}_{k\in\zz}\in F\big(L_A^{p_0}(\ell^{q_0}), L_A^{p_1}(\ell^{q_1})\big)$ such that $f_k=f_k^{\theta}$.
It is known in \cite[2.4.9]{Tr} that for any $0<p_0,p_1,q_0,q_1<\infty$ and $0<\theta<1$
\begin{equation}\label{vectorinterpol}
\big( L_A^{p_0}(\ell^{q_0}), L_A^{p_1}(\ell^{q_1})\big)_{\theta}= L_A^p(\ell^q)
\end{equation} when $1/p=(1-\theta)/p_0+\theta/p_1$ and  $1/q=(1-\theta)/q_0+\theta/q_1$.

\begin{proposition}\label{complexinterpolation}
Let $0<p_0,p_1,q_0,q_1<\infty$, $s_0,s_1 \geq 0$, and $1<r_0,r_1<\infty$.
Suppose that for any $\{g_k\}_{k\in\zz}\in L_A^{p_0}(\ell^{q_0})$ and $\{h_k\}_{k\in \zz}\in L_A^{p_1}(\ell^{q_1})$,
\begin{equation}\label{bound0}
\big\Vert \big\{ m_k^{\vee}\ast g_k\big\}_{k\in\zz}\big\Vert_{L^{p_0}(\ell^{q_0})}\lesssim \mathcal{L}_{s_0}^{r_0}[\mm]\Vert \{g_k\}\Vert_{L^{p_0}(\ell^{q_0})},
\end{equation}
\begin{equation}\label{bound1}
\big\Vert \big\{ m_k^{\vee}\ast h_k\big\}_{k\in\zz}\big\Vert_{L^{p_1}(\ell^{q_1})}\lesssim \mathcal{L}_{s_1}^{r_1}[\mm]\Vert \{h_k\}\Vert_{L^{p_1}(\ell^{q_1})}.
\end{equation}
Then for any $0<\theta<1$ and $p,q,r,s$ satisfying 
\begin{equation}\label{pqcondition}
1/p=(1-\theta)/p_0+\theta/p_1, \qquad 1/q=(1-\theta)/q_0+\theta/q_1,
\end{equation}
\begin{equation}\label{rscondition}
1/r=(1-\theta)/r_0+\theta/r_1,\qquad s=(1-\theta)s_0+\theta s_1,
\end{equation}
and $\{f_k\}_{k\in\zz}\in L_A^p(\ell^q)$, 
we have
\begin{equation*}
\big\Vert \big\{ m_k^{\vee}\ast f_k\big\}_{k\in\zz}\big\Vert_{L^{p}(\ell^{q})}\lesssim \mathcal{L}_{s}^{r}[\mm]\Vert \{f_k\}\Vert_{L^{p}(\ell^{q})}.
\end{equation*}
\end{proposition}

\begin{proof}
Suppose $p,q,r,s$ satisfy (\ref{pqcondition}) and (\ref{rscondition}), and $\{f_k\}_{k\in\zz}\in L_A^p(\ell^q)$.
Then, due to (\ref{vectorinterpol}), for any $\epsilon>0$
there exists $\ff^z=\{f^{z}_k\}\in \big( L_A^{p_0}(\ell^{q_0}), L_A^{p_1}(\ell^{q_1})\big)_{\theta}$ such that $f_k=f_k^{\theta}$ and 
\begin{equation*}
\Vert \ff^z\Vert_{F( L_A^{p_0}(\ell^{q_0}), L_A^{p_1}(\ell^{q_1}))}<\Vert \{f_k\}_{k\in\zz}\Vert_{( L_A^{p_0}(\ell^{q_0}), L_A^{p_1}(\ell^{q_1}))_{\theta}}+\epsilon.
\end{equation*}
Now let 
\begin{equation*}
\sigma_{k,s}:= (I-\Delta)^{s/2}\big(m_k(2^k\cdot)\big)
\end{equation*}
and
\begin{equation*}
\sigma_{k,s}^z:=\big( \mathcal{L}_{s}^{r}[\mm]\big)^{1-r(\frac{1-z}{r_0}+\frac{z}{r_1})}\frac{(1+\theta)^{d/2+1}}{(1+z)^{d/2+1}}(I-\Delta)^{-\frac{s_0(1-z)+s_1z}{2}}\big(|\sigma_{k,s}|^{r(\frac{1-z}{r_0}+\frac{z}{r_1})}e^{i Arg(\sigma_{k,s})}\big)(\cdot/2^k)
\end{equation*} where $Arg(\sigma_{k,s})$ means the argument of  $\sigma_{k,s}$.
Then we note that $\sigma_{k,s}^{\theta}=m_k$ and $F_k^z:=\big(\sigma_{k,s}^z\big)^{\vee}\ast f_k^z$ is a $S'(\rd)$-analytic function in $\Omega$. Moreover,
\begin{align*}
\big\Vert \{m_k^{\vee}\ast f_k\}_{k\in\zz}\big\Vert_{L^p(\ell^q)}&\approx \big\Vert \big\{(\sigma_{k,s}^{\theta})^{\vee}\ast f_k^{\theta}\big\}_{k\in\zz}\big\Vert_{ ( L_A^{p_0}(\ell^{q_0}), L_A^{p_1}(\ell^{q_1}))_{\theta}}\\
&=\big\Vert \{F_k^{\theta}\}_{k\in\zz}\big\Vert_{( L_A^{p_0}(\ell^{q_0}), L_A^{p_1}(\ell^{q_1}))_{\theta}}\leq \big\Vert \{F_k^{z}\}_{k\in\zz}\big\Vert_{F( L_A^{p_0}(\ell^{q_0}), L_A^{p_1}(\ell^{q_1}))}\\
&=\max{\Big(\sup_{t\in\rr}{\big\Vert \{F_k^{it}\}_{k\in\zz}\big\Vert_{L^{L^{p_0}}(\ell^{q_0})}},\sup_{t\in\rr}{\big\Vert \{F_k^{1+it}\}_{k\in\zz}\big\Vert_{L^{p_1}(\ell^{q_1})}} \Big)}.
\end{align*}
From (\ref{bound0}),
\begin{align*}
\big\Vert \{F_k^{it}\}_{k\in\zz}\big\Vert_{L^{p_0}(\ell^{q_0})}&=\big\Vert \big\{ (\sigma_{k,s}^{it})^{\vee}\ast f_k^{it}\big\}_{k\in\zz}\big\Vert_{L^{p_0}(\ell^{q_0})}\\
&\lesssim \sup_{j\in\rr}{\big\Vert \sigma_{j,s}^{it}(2^j\cdot)\big\Vert_{L_{s_0}^{r_0}(\rd)}}\Vert \{f_k^{it}\}_{k\in\zz}\Vert_{L^{p_0}(\ell^{q_0})}\\
&\lesssim\sup_{j\in\rr}{\big\Vert \sigma_{j,s}^{it}(2^j\cdot)\big\Vert_{L_{s_0}^{r_0}(\rd)}}\big( \Vert \{f_k\}_{k\in\zz}\Vert_{(L^{p_0}(\ell^{q_0}),L^{p_1}(\ell^{q_1}))_{\theta}}+\epsilon\big)
\end{align*}
and similarly, thanks to (\ref{bound1}),
\begin{equation*}
\big\Vert \{F_k^{1+it}\}_{k\in\zz}\big\Vert_{L^{p_1}(\ell^{q_1})}\lesssim\sup_{j\in\rr}{\big\Vert \sigma_{j,s}^{1+it}(2^j\cdot)\big\Vert_{L_{s_1}^{r_1}(\rd)}}\big( \Vert \{f_k\}_{k\in\zz}\Vert_{(L^{p_0}(\ell^{q_0}),L^{p_1}(\ell^{q_1}))_{\theta}}+\epsilon\big).
\end{equation*}
Therefore, once we prove
\begin{equation}\label{reduction}
\big\Vert \sigma_{j,s}^{it}(2^j\cdot)\big\Vert_{L_{s_0}^{r_0}(\rd)}, \big\Vert \sigma_{j,s}^{1+it}(2^j\cdot)\big\Vert_{L_{s_1}^{r_1}(\rd)}\lesssim \mathcal{L}_{s}^{r}[\mm], \qquad \text{ uniformly in }~j\in\zz,
\end{equation}
then we are done by using (\ref{vectorinterpol}) and taking $\epsilon\to 0$.

Let us prove (\ref{reduction}). 
By using H\"ormander's multiplier theorem, $\big\Vert \sigma_{j,s}^{it}(2^j\cdot)\big\Vert_{L_{s_0}^{r_0}(\rd)}$ is controlled by a constant times
\begin{align*}
& \big(\mathcal{L}_{s}^{r}[\mm] \big)^{1-\frac{r}{r_0}}\frac{1}{(1+|t|)^{d/2+1}}\Big\Vert (I-\Delta)^{\frac{it(s_0-s_1)}{2}}\Big(|\sigma_{j,s}|^{\frac{r}{r_0}-itr(\frac{1}{r_0}-\frac{1}{r_1})}e^{i Arg(\sigma_{j,s})} \Big)\Big\Vert_{L^{r_0}(\rd)}\\
&\lesssim  \big(\mathcal{L}_{s}^{r}[\mm] \big)^{1-\frac{r}{r_0}} \big\Vert |\sigma_{j,s}|^{r/r_0}\big\Vert_{L^{r_0}(\rd)}=\big(\mathcal{L}_{s}^{r}[\mm] \big)^{1-\frac{r}{r_0}} \big\Vert \sigma_{j,s}\big\Vert_{L^{r}(\rd)}^{r/r_0}\leq \mathcal{L}_{s}^{r}[\mm].
\end{align*}
On the other hand, $\big\Vert \sigma_{j,s}^{1+it}(2^j\cdot)\big\Vert_{L_{s_1}^{r_1}(\rd)}$ is less than a constant multiple of
\begin{align*}
& \big(\mathcal{L}_{s}^{r}[\mm] \big)^{1-\frac{r}{r_1}}\frac{1}{(1+|t|)^{d/2+1}}\Big\Vert (I-\Delta)^{\frac{it(s_0-s_1)}{2}}\Big(|\sigma_{j,s}|^{\frac{r}{r_1}-itr(\frac{1}{r_0}-\frac{1}{r_1})}e^{i Arg(\sigma_{j,s})} \Big)\Big\Vert_{L^{r_1}(\rd)}\\
&\lesssim  \big(\mathcal{L}_{s}^{r}[\mm] \big)^{1-\frac{r}{r_1}} \big\Vert |\sigma_{j,s}|^{r/r_1}\big\Vert_{L^{r_1}(\rd)}=\big(\mathcal{L}_{s}^{r}[\mm] \big)^{1-\frac{r}{r_1}} \big\Vert \sigma_{j,s}\big\Vert_{L^{r}(\rd)}^{r/r_1}\leq \mathcal{L}_{s}^{r}[\mm],
\end{align*} which finishes the proof of (\ref{reduction}).

\end{proof}

\section{The Key Lemma}\label{keylemmasection}

Suppose that (\ref{assumptionm}) holds. Then for $1<r_0<r_1<\infty$ and $s\geq 0$ we have
\begin{equation}\label{membedding}
\Vert m_k(2^k\cdot)\Vert_{L^{r_0}_s(\rd)}\lesssim \Vert m_k(2^k\cdot)\Vert_{L^{r_1}_s(\rd)}.
\end{equation}  
The proof of this will be given in Appendix.
Now the principal ingredient in the proof of Theorem \ref{main1} and \ref{main2} is the following lemma:
\begin{lemma}\label{basiclemma}
Suppose $0<p\leq \infty$ and $k\in\mathbb{Z}$. Suppose $f_k\in \mathcal{E}(2^{k-2})$  and $\{m_k\}_{k\in\mathbb{Z}}$ satisfies $(\ref{assumptionm})$.
Then for 
\begin{equation*}
\big| d/p-d/2\big|<s<d/\min{(1,p)} \quad \text{ and }~\quad   r>\tau^{(s,p)},
\end{equation*} we have
\begin{equation*}
\big\Vert m_k^{\vee}\ast  f_k\big\Vert_{L^p(\rd)}\lesssim \big\Vert m_k(2^k\cdot)\big\Vert_{L_s^{r}(\rd)}\Vert f_k\Vert_{L^p(\rd)}\quad \text{uniformly in }k.
\end{equation*} 
\end{lemma}

\begin{proof}
This is trivial when $1<p<\infty$, due to Theorem \ref{thma}, and thus we are mainly concerned with the case $0<p\leq 1$ or $p=\infty$, assuming  $d/\min{(1,p)}-d/2<s<d/\min{(1,p)}$, which
implies that $1<\tau^{(s,p)}<2$. Furthermore, thanks to (\ref{membedding}), we may also assume that $\tau^{(s,p)}<r<2$.

When $p=1$ or $p=\infty$, it follows immediately from Young's inequality that 
\begin{equation*}
\big\Vert m_k^{\vee}\ast  f_k\big\Vert_{L^p(\rd)}\lesssim \Vert m_k^{\vee}\Vert_{L^1(\rd)} \Vert f_k\Vert_{L^p(\rd)}.
\end{equation*}
On the other hand, using a dilation, H\"older's inequality with $r>1$, and the Hausdorff-Young inequality with $1<r<2$, we obtain
\begin{align*}
\Vert m_k^{\vee}\Vert_{L^1(\rd)}&=\big\Vert \big( m_k(2^k\cdot)\big)^{\vee}\big\Vert_{L^1(\rd)}\lesssim \big\Vert \big(1+4\pi^2|\cdot|^2 \big)^{s/2}\big(m_k(2^k\cdot) \big)^{\vee} \big\Vert_{L^{r'}(\rd)}\\
 &\lesssim \big\Vert m_k(2^k\cdot)\big\Vert_{L^r_s(\rd)},
\end{align*} which ends the argument.

For $0<p<1$, Bernstein's inequality ( see \cite[1.3.2]{Tr} ) proves that
\begin{equation*}
\big\Vert ( m_k )^{\vee}\ast f_k\big\Vert_{L^p(\rd)} \lesssim 2^{kd(1/p-1)}\big\Vert ( m_k )^{\vee}\big\Vert_{L^p(\rd)}\Vert f_k\Vert_{L^p(\rd)}
\end{equation*}
and then using a dilation, H\"older's inequality with $t:=\frac{1}{1-p+p/r}>1$, and the Hausdorff-Young inequality with $1<r<2$, we have
\begin{align*}
2^{kd(1/p-1)}\Vert m_k ^{\vee}\Vert_{L^p(\rd)}&=\big\Vert \big( m_k(2^k\cdot)\big)^{\vee}\big\Vert_{L^p(\rd)}\lesssim \Big\Vert \big| \big( 1+4\pi^2|\cdot|^2\big)^{s/2}\big(m_k(2^k\cdot) \big)^{\vee}\big|^p\Big\Vert_{L^{t'}(\rd)}^{1/p}\\
&=\Big\Vert  \big( 1+4\pi^2|\cdot|^2\big)^{s/2}\big(m_k(2^k\cdot) \big)^{\vee}\Big\Vert_{L^{pt'}(\rd)}\lesssim \big\Vert m_k(2^k\cdot)\big\Vert_{L_s^r(\rd)}
\end{align*} since $r'=pt'$.
This completes the proof.     
\end{proof}

\section{Proof of Theorem \ref{main2}}

Let $|d/q-d/2|<s<d/\min{(1,q)}$ and $r>\tau^{(s,q)}$.
Suppose $\nu\geq \mu$ and $P\in\mathcal{D}_{\nu}$ (i.e. $\ell(P)=2^{-\nu}\leq 2^{-\mu}$).
Let  $P^{*}=9P$ denote the concentric dilate of $P$ by a factor of $9$. Note that $P^*$ is a union of some dyadic cubes near $P$.
Then we decompose
\begin{align*}
\Big(\frac{1}{|P|}\int_P{\sum_{k=\nu}^{\infty}{\big| m_k^{\vee}\ast f_{k}(x)\big|^q}}dx \Big)^{1/q}&\lesssim  \Big(\frac{1}{|P|}\int_P{\sum_{k=\nu}^{\infty}{\big|m_{k}^{\vee}\ast \big(\chi_{P^*} f_{k}\big)(x) \big|^q}}dx\Big)^{1/q}\\
&\relphantom{=}+\Big(\frac{1}{|P|}\int_P{\sum_{k=\nu}^{\infty}{\big|m_{k}^{\vee}\ast \big(\chi_{(P^*)^c} f_{k}\big)(x) \big|^q}}dx\Big)^{1/q}\\
&=:\mathcal{U}_P+\mathcal{V}_P.
\end{align*}

We observe that, due to (\ref{assumptionm}), 
\begin{equation}\label{psiast}
m_{k}^{\vee}\ast \big( \chi_{P^*}f_k\big)=m_{k}^{\vee}\ast {\Psi_{k+1}}\ast\big( \chi_{P^*}f_k\big)
\end{equation}
and then $\mathcal{U}_P$ is estimated by
\begin{equation*}
 \Big( \frac{1}{|P|}\sum_{k=\nu}^{\infty}{\big\Vert m_{k}^{\vee}\ast \Psi_{k+1}\ast \big( \chi_{P^*}f_k\big)\big\Vert_{L^q(\rd)}^q}\Big)^{1/q}\lesssim \mathcal{L}_{s}^{r}[\mm]\Big( \frac{1}{|P|}\sum_{k=\nu}^{\infty}{\big\Vert \Psi_{k+1}\ast \big(\chi_{P^*}f_k\big)\big\Vert_{L^q(\rd)}^q}\Big)^{1/q},
\end{equation*} due to Lemma \ref{basiclemma}.
We now claim that for any $\sigma>0$
\begin{equation}\label{lqest}
\big\Vert \Psi_{k+1}\ast \big(\chi_{P^*}f_k\big)\big\Vert_{L^q(\rd)}\lesssim_{\sigma} \Big(\int_{P^*}{\big(\mathfrak{M}_{\sigma,2^k}f_k(y)\big)^q}dy\Big)^{1/q}.
\end{equation}
This follows immediately from Young's inequality if $q\geq 1$.
For $0<q<1$,
we write
\begin{align*}
 \big\Vert \Psi_{k+1}\ast \big(\chi_{P^*}f_k\big)\big\Vert_{L^q(\rd)}^q&=\sum_{Q\in\mathcal{D}_k,Q\subset P^*}{\big\Vert \Psi_{k+1}\ast \big(\chi_{Q}f_k\big)\big\Vert_{L^q(\rd)}^q}\\
&\leq \sum_{Q\in\mathcal{D}_k,Q\subset P^*}{\Vert f_k\Vert_{L^{\infty}(Q)}^q}\int_{\rd}{\Big(\int_Q{|\Psi_{k+1}(x-y)|}dy \Big)^q}dx.
\end{align*}
The integral in the preceding expression can be estimated, using H\"older's inequality with $1/q>1$, by
\begin{equation*}
\Big(\int_{\rd}{\frac{1}{(1+2^k|x-c_Q|)^{M/(1-q)}}}dx \Big)^{1-q}\Big(\int_{\rd}{\int_{Q}{(1+2^k|x-c_Q|)^{M/q}\big|\Psi_{k+1}(x-y) \big|}dy}dx \Big)^{q},
\end{equation*}
which is clearly smaller than a constant multiple of $2^{-kd}$ for sufficiently large $M>0$.
This, together with (\ref{infmax}), yields that
\begin{align*}
 \big\Vert \Psi_{k+1}\ast \big(\chi_{P^*}f_k\big)\big\Vert_{L^q(\rd)}^q&\lesssim\sum_{Q\in\mathcal{D}_k,Q\subset P^*}{2^{-kd}\inf_{y\in Q}{\big(\mathfrak{M}_{\sigma,2^k}f_k(y)\big)^q}}\\
&\leq \sum_{Q\in\mathcal{D}_k,Q\subset P^*}{\int_Q{\big(\mathfrak{M}_{\sigma,2^k}f_k(y)\big)^q}dy}=\int_{P^*}{\big(\mathfrak{M}_{\sigma,2^k}f_k(y)\big)^q}dy
\end{align*} 
and we finally arrive at the desired estimate (\ref{lqest}).
Therefore we have
\begin{align}
\mathcal{U}_P&\lesssim  \mathcal{L}_{s}^{r}[\mm]\Big(\frac{1}{|P|}\int_{P^*}{\sum_{k=\nu}^{\infty}{\big(\mathfrak{M}_{\sigma,2^k}f_k(y)\big)^q}}dy \Big)^{1/q}\label{recall1}\\
&\lesssim  \mathcal{L}_{s}^{r}[\mm]\sup_{R\in\mathcal{D}_{\nu}}\Big(\frac{1}{|R|}\int_{R}{\sum_{k=\nu}^{\infty}{\big(\mathfrak{M}_{\sigma,2^k}f_k(y)\big)^q}}dy \Big)^{1/q}.\nonumber
\end{align}
Choosing $\sigma>d/q$ and applying the maximal inequality (\ref{maximal2}), we conclude that
\begin{equation*}
\mathcal{U}_P\lesssim  \mathcal{L}_{s}^{r}[\mm] \sup_{R\in\mathcal{D}: \ell(R)\leq 2^{-\mu}}{\Big( \frac{1}{|R|}\int_R{\sum_{k=-\log_2{\ell(R)}}^{\infty}{\big|f_k(x) \big|^q}}dx\Big)^{1/q}}.
\end{equation*}

To estimate $\mathcal{V}_P$ we note that $r>\tau^{(s,q)}$ implies that $s-d/r>d/\min{(1,q)}-d$ and there exists $\epsilon>0$ so that $s-\epsilon-d/r>d/\min{(1,q)}-d\geq 0$.
Then we see that for $x\in P$
\begin{align*}
\big| m_{k}^{\vee}\ast \big(\chi_{(P^*)^c}f_k \big)(x)\big| &\leq \int_{|z|\gtrsim \ell(P)}{|m_{k}^{\vee}(z)||f_k(x-z)|}dz\\
&\leq \mathfrak{M}_{\epsilon,2^k}f_k(x)\int_{|z|\gtrsim \ell(P)}{\big(1+2^k|z|\big)^{\epsilon}|m_{k}^{\vee}(z)|}dz
\end{align*}
and the integral is less than a constant times
\begin{align*}
&\Big( \int_{|z|\gtrsim 2^k \ell(P)}{\frac{1}{|z|^{(s-\epsilon)r}}}dz\Big)^{1/r}\big\Vert (1+4\pi^2|\cdot|^2)^{s/2}\big(m_k(2^k\cdot)\big)^{\vee}\big\Vert_{L^{r'}(\rd)}\lesssim 2^{-(k-\nu)(s-\epsilon-d/r)}\mathcal{L}_{s}^{r}[\mm]
\end{align*} by applying H\"older's inequality and the Hausdorff-Young inequality.
This proves that
\begin{align}
\mathcal{V}_P&\lesssim \mathcal{L}_{s}^{r}[\mm]\Big( \frac{1}{|P|}\int_P{\sum_{k=\nu}^{\infty}{2^{-q(k-\nu)(s-\epsilon-d/r)}\big( \mathfrak{M}_{\epsilon,2^k}f_k(x)\big)^{q}}}dx\Big)^{1/q}\label{estvp}\\
&\lesssim \mathcal{L}_{s}^{r}[\mm]\big\Vert \big\{\mathfrak{M}_{\epsilon,2^k}f_k\big\}_{k\geq \nu}\big\Vert_{L^{\infty}(l^{\infty})}\lesssim \mathcal{L}_{s}^{r}[\mm]\big\Vert \big\{f_k\big\}_{k\geq \nu}\big\Vert_{L^{\infty}(l^{\infty})}\nonumber\\
&\lesssim \mathcal{L}_{s}^{r}[\mm]\sup_{R\in\mathcal{D}: \ell(R)\leq 2^{-\nu}}\Big( \frac{1}{|R|}\int_P{\sum_{k=-\log_2{\ell(R)}}^{\infty}{\big|f_k(x)\big|^{q}}}dx\Big)^{1/q} \nonumber
\end{align} where the maximal inequality (\ref{maximal1})  and the embedding (\ref{embeddinginfty}) are applied. 

By taking the supremum of $\mathcal{U}_P$ and $\mathcal{V}_P$ over all dyadic cubes $P$ whose side length is less or equal to $2^{-\mu}$, the proof of Theorem \ref{main2} is complete.

\section{Proof of Theorem \ref{main1}}
A straightforward application of Lemma \ref{basiclemma} proves the special case $0<p=q\leq \infty$ and therefore we work only with the case $p\not= q$ and $0<p<\infty$.

\subsection{\textbf{The case $0<p\leq 1$ and $p<q\leq \infty$}}
Assume $d/p-d/2<s<d/p$. Then $1<\tau^{(s,p)}<2$ and we may assume $\tau^{(s,p)}<r<2$ because of (\ref{membedding}).
 According to Lemma \ref{discrete1} and Lemma \ref{decomhardy},
if $Supp(\widehat{f_k})\subset \{\xi:|\xi|\leq 2^{k-1}\}$ for each $k\in\mathbb{Z}$, then
there exist $\{b_Q\}_{Q\in\mathcal{D}}\in \dot{f}_p^{0,q}$, a sequence of scalars $\{\lambda_j\}$, and a sequence of $\infty$-atoms $\{r_{j,Q}\}$ for $\dot{f}_p^{0,q}$ such that
\begin{equation*}
f_k(x)=\sum_{Q\in\mathcal{D}_k}{b_Q\Psi^Q(x)}=\sum_{j=1}^{\infty}{\lambda_j \sum_{Q\in\mathcal{D}_k}{ r_{j,Q}\Psi^Q (x)  }}, \quad k\in\mathbb{Z},
\end{equation*}
and
\begin{equation*}
\Big( \sum_{j=1}^{\infty}{|\lambda_j|^p}\Big)^{1/p}\lesssim \Vert b\Vert_{\dot{f}_p^{0,q}}\lesssim \big\Vert \big\{ f_k\big\}_{k\in\mathbb{Z}}\big\Vert_{L^p(\ell^q)}.
\end{equation*}
Then by applying $\ell^p\hookrightarrow \ell^1$ and Minkowski's inequality with $q/p>1$, we have
\begin{align*}
\big\Vert   \big\{m_k^{\vee}\ast f_k \big\}_{k\in\mathbb{Z}}\big\Vert_{L^p(\ell^q)} &\lesssim  \big( \sum_{j=1}^{\infty}{|\lambda_j|^p}\big)^{1/p}\sup_{n\geq 1}{\Big\Vert \Big\{ m_k^{\vee}\ast \Big(\sum_{Q\in\mathcal{D}_k}{r_{n,Q}\Psi^Q} \Big)\Big\}_{k\in\mathbb{Z}}\Big\Vert_{L^p(\ell^q)}}\\
&\lesssim \big\Vert \big\{ f_k\big\}_{k\in\mathbb{Z}}\big\Vert_{L^p(\ell^q)}\sup_{n\geq 1}{\Big\Vert \Big\{ m_k^{\vee}\ast \Big(\sum_{Q\in\mathcal{D}_k}{r_{n,Q}\Psi^Q} \Big)\Big\}_{k\in\mathbb{Z}}\Big\Vert_{L^p(\ell^q)}}.
\end{align*} 
Therefore, it suffices to show that the supremum in the above expression is dominated by a constant times $\mathcal{L}_{s}^{r}[\mm]$, which is equivalent to
\begin{equation*}
\big\Vert \big\{ m_k^{\vee}\ast A_{Q_0,k}\big\}_{k\in\mathbb{Z}}\big\Vert_{L^p(\ell^q)}\lesssim \mathcal{L}_{s}^{r}[\mm]\quad \text{ uniformly in }~  Q_0
\end{equation*} where $\{r_Q\}$ is an $\infty$-atom for $\dot{f}_p^{0,q}$ associated with $Q_0\in\mathcal{D}$ and 
\begin{equation*}
A_{Q_0,k}(x):=\sum_{Q\in\mathcal{D}_k,Q\subset Q_0}{r_{Q}\Psi^Q(x)}.
\end{equation*}

Suppose $Q_0\in\mathcal{D}_{\nu}$ for some $\nu\in\mathbb{Z}$. 
Then the condition $Q\subset Q_0$ ensures that $A_{Q_0,k}$ vanishes unless $\nu\leq k$, and thus our actual goal now is to prove
\begin{equation}\label{ouractualgoal}
\big\Vert \big\{ m_k^{\vee}\ast A_{Q_0,k}\big\}_{k\geq \nu}\big\Vert_{L^p(\ell^q)}\lesssim \mathcal{L}_{s}^{r}[\mm] \quad \text{ uniformly in }~ \nu ~\text{ and }~ Q_0.
\end{equation}

We observe that for $x\in\mathbb{R}^d$
\begin{equation}\label{estlq}
 \big\Vert \big\{ |r_Q||Q|^{-1/2}\chi_Q(x)\big\}_{Q\subset Q_0}\big\Vert_{\ell^q}\leq |Q_0|^{-1/p}
\end{equation}
and for $0<t<\infty$
\begin{equation}\label{aqest}
\Vert A_{Q_0,k}\Vert_{L^t(\rd)}\lesssim \Big\Vert \sum_{Q\in\mathcal{D}_k, Q\subset Q_0}{|r_Q||Q|^{-1/2}\chi_Q}\Big\Vert_{L^t(\rd)}\leq |Q_0|^{-1/p+1/t}
\end{equation} by using the argument in (\ref{converse2}) and the estimate (\ref{infdef}). 
Moreover, 
\begin{equation*}
Supp(\widehat{A_{Q_0,k}})=Supp(\widehat{\Psi_k})\subset \big\{\xi: |\xi|\leq 2^k \big\}.
\end{equation*}

Let $Q_0^*$ and $Q_0^{**}$ denote the concentric dilates of $Q_0$ with side length $9\ell(Q_0)$ and $81\ell(Q_0)$, respectively.
Then we write
\begin{align}\label{2goal}
\big\Vert \big\{ m_k^{\vee}\ast A_{Q_0,k}\big\}_{k\geq \nu}\big\Vert_{L^p(\ell^q)}&\lesssim \Big( \int_{Q_0^{**}}{\big\Vert \big\{ m_k^{\vee}\ast A_{Q_0,k}(x)\big\}_{k\geq \nu}\big\Vert_{\ell^q}^p}dx\Big)^{1/p}\nonumber\\
&\relphantom{=}+\Big( \int_{(Q_0^{**})^c}{\big\Vert \big\{ m_k^{\vee}\ast A_{Q_0,k}(x)\big\}_{k\geq \nu}\big\Vert_{\ell^q}^p}dx\Big)^{1/p}.
\end{align}
Using H\"older's inequality and Lemma \ref{basiclemma} with $\tau^{(s,q)}\leq \tau^{(s,p)}<r$ and
\begin{equation*}
\big|d/q-d/2\big|<s-(d/p-d/\min{(1,q)})<d/\min{(1,q)},
\end{equation*} 
the first one is controlled by
\begin{align*}
&|Q_0^{**}|^{1/p-1/q}\big\Vert \big\{ m_k^{\vee}\ast A_{Q_0,k}\big\}_{k\geq \nu}\big\Vert_{L^q(\ell^q)}\\
&\lesssim \sup_{l\in\zz}{\big\Vert m_l(2^l\cdot)\big\Vert_{L_{s-(d/p-d/\min{(1,q)})}^{r}(\rd)}} |Q_0|^{1/p-1/q}\big\Vert \big\{ A_{Q_0,k}\big\}_{k\geq \nu}\big\Vert_{L^q(\ell^q)}
\end{align*}
and we see that, from (\ref{converse2}) and (\ref{estlq}),
\begin{equation*}
\big\Vert \big\{ A_{Q_0,k}\big\}_{k\geq \nu}\big\Vert_{L^q(\ell^q)}\lesssim \big\Vert \{r_Q\}_{Q\in\mathcal{D},Q\subset Q_0}\big\Vert_{\dot{f}_q^{0,q}} \lesssim |Q_0|^{-1/p+1/q}.
\end{equation*}
Now using the embedding $L_s^r(\rd)\hookrightarrow L_{s-(d/p-d/\min{(1,q)})}^{r}(\rd)$, we obtain
\begin{equation*}
\sup_{l\in\zz}{\big\Vert m_l(2^l\cdot)\big\Vert_{L_{s-(d/p-d/\min{(1,q)})}^{r}(\rd)}}\lesssim \mathcal{L}_s^r[\mm],
\end{equation*} which finishes the proof of 
\begin{equation*}
\Big( \int_{Q_0^{**}}{\big\Vert \big\{ m_k^{\vee}\ast A_{Q_0,k}(x)\big\}_{k\geq \nu}\big\Vert_{\ell^q}^p}dx\Big)^{1/p}\lesssim \mathcal{L}_s^r[\mm].
\end{equation*}

To handle the term (\ref{2goal}) we make use of the embedding $\ell^p\hookrightarrow \ell^q$ to obtain
\begin{equation*}
(\ref{2goal})\leq \Big(\sum_{k=\nu}^{\infty}{\big\Vert m_k^{\vee}\ast A_{Q_0,k}\big\Vert_{L^p((Q_0^{**})^c)}^p} \Big)^{1/p}.
\end{equation*} 
Then, writing
\begin{equation*}
\big\Vert m_k^{\vee}\ast A_{Q_0,k}\big\Vert_{L^p((Q_0^{**})^c)}^p\leq \big\Vert m_k^{\vee}\ast \big(A_{Q_0,k}\chi_{Q_0^*}\big)\big\Vert_{L^p((Q_0^{**})^c)}^p+\big\Vert m_k^{\vee}\ast \big(A_{Q_0,k}\chi_{(Q_0^*)^c}\big)\big\Vert_{L^p((Q_0^{**})^c)}^p,
\end{equation*}
the proof of (\ref{ouractualgoal}) will be complete once we establish the estimates
 that for some $\delta>0$
\begin{equation}\label{i}
\big\Vert m_k^{\vee}\ast \big(A_{Q_0,k}\chi_{Q_0^*}\big)\big\Vert_{L^p((Q_0^{**})^c)}\lesssim 2^{-\delta(k-\nu)}\mathcal{L}_s^r[\mm],
\end{equation} 
\begin{equation}\label{ii}
\big\Vert m_k^{\vee}\ast \big(A_{Q_0,k}\chi_{(Q_0^*)^c}\big)\big\Vert_{L^p((Q_0^{**})^c)}\lesssim 2^{-\delta(k-\nu)}\mathcal{L}_s^r[\mm].
\end{equation}

It follows from the embedding $\ell^p\hookrightarrow \ell^1$ that
\begin{align*}
&\big\Vert m_k^{\vee}\ast \big(A_{Q_0,k}\chi_{Q_0^*}\big)\big\Vert_{L^p((Q_0^{**})^c)}\\
&\leq \Big({\sum_{Q\in\mathcal{D}_k,Q\subset Q_0^{*}}{\int_{({Q_0}^{**})^c }{\big| m_k^{\vee}\ast\big(A_{Q_0,k}\chi_{Q}\big)(x) \big|^p}dx}} \Big)^{1/p}\\
&\leq \Big({\sum_{Q\in\mathcal{D}_k,Q\subset Q_0^*}{\Vert A_{Q_0,k}\Vert_{L^{\infty}(Q)}^p\int_{({Q_0}^{**})^c}{\Big( \int_{{Q}}{|m_k^{\vee}(x-y)|}dy\Big)^p}dx }}\Big)^{1/p}.
\end{align*}
We notice that the assumption $r>\tau^{(s,p)}$ is equivalent to $s>d/r+d/p-d$ and therefore there exists $M>d(1-p)$ such that $s>d/r+M/p>d/r+d/p-d$.
Recall that $x_Q$ denotes the left lower corner of $Q\in\mathcal{D}$ and observe that for $Q\subset Q_0^*$
\begin{align*}
& \int_{({Q_0}^{**})^c}{\Big( \int_{{Q}}{|m_k^{\vee}(x-y)|}dy\Big)^p}dx\\
&\lesssim 2^{-kM}\ell(Q_0)^{-M+d(1-p)}\Big(\int_{{Q}}{\int_{({Q_0}^{**})^c}{\big(1+2^k|x-x_Q|\big)^{M/p}\big| m_k^{\vee}(x-y)\big|}dx}dy\Big)^p\\
&\lesssim 2^{-k(M+pd)}\ell(Q_0)^{-M+d(1-p)}\Big( \int_{\mathbb{R}^d}{\big(1+2^k|y| \big)^{M/p}|m_k^{\vee}(y)|}dy\Big)^p
\end{align*}
where we utilized H\"older's inequality if $0<p<1$ and the fact that $|x-x_Q|\lesssim |x-y|$ for $x\in (Q_0^{**})^c$ and $y\in Q\subset Q_0^*$.
Moreover,  H\"older's inequality with $r>1$ and the Hausdorff-young inequality yield that
 \begin{align*}
 \Big( \int_{\mathbb{R}^d}{\big(1+2^k|y| \big)^{M/p}|m_k^{\vee}(y)|}dy\Big)^p&=\Big( \int_{\rd}{\big(1+|y|^2\big)^{M/p}\big| \big(m_k(2^k\cdot)\big)^{\vee}(y)\big|}dy\Big)^p\\
 &\lesssim \big\Vert (1+4\pi^2|\cdot|^2)^{s/2}|(m_k(2^k\cdot))^{\vee}|\big\Vert_{L^{r'}(\rd)}^{p}\\
 &\lesssim \mathcal{L}_{s}^{r}[\mm].
 \end{align*}
 Furthermore, (\ref{infmax}) proves that for $\sigma>d/p$
 \begin{equation*}
 \Vert A_{Q_0,k}\Vert_{L^{\infty}(Q)}\lesssim \inf_{y\in Q}{\mathfrak{M}_{\sigma,2^k}A_{Q_0,k}(y)}\lesssim 2^{kd/p}\big\Vert \mathfrak{M}_{\sigma,2^k}A_{Q_0,k}\big\Vert_{L^p(Q)}.
 \end{equation*}
 Consequently,
\begin{align*}
\big\Vert m_k^{\vee}\ast \big(A_{Q_0,k}\chi_{Q_0^*}\big)\big\Vert_{L^p((Q_0^{**})^c)}&\lesssim 2^{-(k-\nu)(M/p-(d/p-d))}\mathcal{L}_{s}^{r}[\mm]\big\Vert \mathfrak{M}_{\sigma,2^k}A_{Q_0,k}\big\Vert_{L^p(Q_0)}\\
 &\lesssim 2^{-(k-\nu)(M/p-(d/p-d))}\mathcal{L}_{s}^{r}[\mm]
\end{align*} where we applied (\ref{maximal1}) with $\sigma>d/p$ and (\ref{aqest}) to obtain $\big\Vert \mathfrak{M}_{\sigma,2^k}A_{Q_0,k}\big\Vert_{L^p(Q_0)}\lesssim 1$.
Then (\ref{i}) follows with $\delta=M/p-(d/p-d)>0$.

To verify (\ref{ii}) we see that, similar to (\ref{psiast}), under the assumption (\ref{assumptionm}),
\begin{equation*}
m_k^{\vee}\ast \big(A_{Q_0,k}\chi_{({Q_0^*})^c}\big)=m_k^{\vee}\ast {\Psi_{k+1}}\ast \big(A_{Q_0,k}\chi_{({Q_0^*})^c}\big)
\end{equation*}  and, it follows from Lemma \ref{basiclemma} that
\begin{equation*}
\big\Vert m_k^{\vee}\ast \big(A_{Q_0,k}\chi_{({Q_0^*})^c}\big)\big\Vert_{L^p(\rd)}\lesssim \mathcal{L}_{s}^{r}[\mm]\big\Vert {\Psi_{k+1}}\ast \big(A_{Q_0,k}\chi_{({Q_0^*})^c}\big)\big\Vert_{L^p(\rd)}.
\end{equation*}
In addition, for sufficiently large $L>0$,
\begin{align*}
&\big\Vert \Psi_{k+1}\ast (A_{Q_0,k}\chi_{({Q_0^*})^c})\big\Vert_{L^p(\rd)}\\
&\lesssim_L \Big(\int_{\mathbb{R}^d}{\Big({\sum_{Q\in\mathcal{D}_k,Q\subset Q_0}|r_Q||Q|^{-1/2}\int_{({Q_0^*})^c}\big|\Psi_{k+1}(x-y)\big|{\dfrac{1}{(1+2^k|y-x_Q|)^{2L}}}}dy \Big)^p}dx\Big)^{1/p}\\
&\lesssim 2^{-kL}\Big(\sum_{Q\in\mathcal{D}_k,Q\subset Q_0}{|r_Q||Q|^{-1/2}} \Big)\Big(\int_{\mathbb{R}^d}{\Big(\int_{({Q_0^*})^c}{\dfrac{|\Psi_{k+1}(x-y)|}{|y-x_{Q_0}|^L}}}dy \Big)^pdx\Big)^{1/p}
\end{align*} 
because $|y-x_Q|\gtrsim \ell(Q_0)$  and
\begin{equation*}
\frac{1}{(1+2^k|y-x_Q|)^{2L}}\lesssim \big(2^{k}\ell(Q_0)\big)^{-L}\frac{\big(1+2^k|x_Q-x_{Q_0}| \big)^L}{\big(1+2^k|y-x_{Q_0}| \big)^L}\lesssim \frac{1}{\big(2^k|y-x_{Q_0}| \big)^L}
\end{equation*}for $y\in (Q_0^*)^c$ and $Q\subset Q_0$.
Due to (\ref{estlq}), we have
\begin{equation*}
\sum_{Q\in\mathcal{D}_k,Q\subset Q_0}{|r_Q||Q|^{-1/2}}\leq 2^{\nu d(1/p-1)}2^{kd}
\end{equation*}
and, using H\"older's inequality (if $p<1$), we obtain that
\begin{align*}
& \Big(\int_{\mathbb{R}^d}{\Big(\int_{({Q_0^*})^c}{\dfrac{|\Psi_{k+1}(x-y)|}{|y-x_{Q_0}|^L}}}dy \Big)^pdx\Big)^{1/p}\\
&\lesssim_N 2^{-kd(1/p-1)}\int_{(Q_0^*)^c}{\frac{1}{|y-x_{Q_0}|^L}\int_{\mathbb{R}^d}{\big(1+2^k|x-x_{Q_0}| \big)^{N/p}\big| \Psi_{k+1}(x-y)\big|}dx}dy\\
&\lesssim 2^{-kd(1/p-1)}2^{kN/p}\int_{(Q_0^*)^c}{\frac{1}{|y-x_{Q_0}|^{L-N/p}}}dy\\
&\lesssim_{L,N} 2^{-kd(1/p-1)}2^{kN/p}2^{\nu(L-N/p-d)}
\end{align*} for $N>d(1-p)$ and $L-N/p>d$.

Finally, we have
\begin{equation*}
\big\Vert \Psi_{k+1}\ast (A_{Q_0,k}\chi_{({Q_0^*})^c})\big\Vert_{L^p(\rd)}\lesssim 2^{-(k-\nu)(L-N/p+d/p-2d)}
\end{equation*}
and this leads to (\ref{ii}) with $\delta=L-N/p+d/p-2d>0$.

\subsection{\textbf{The case $0<q\leq 1$ and $q<p<\infty$ }}
Assume $s>d/\min{(1,q)}-d/2$ and $r>\tau^{(s,q)}$.
As in the proof of Theorem \ref{main2}, we select $\epsilon>0$ so that $s-\epsilon-d/r>d/\min{(1,q)}-d$.

We first consider the case $p>d/\epsilon$.
In view of Lemma \ref{charactersharp} we can write
\begin{equation*}
\big\Vert \big\{ m_k^{\vee}\ast f_k\big\}_{k\in\mathbb{Z}} \big\Vert_{L^p(\ell^q)}\lesssim \Big\Vert \sup_{P:x\in P\in\mathcal{D}}{\Big(\frac{1}{|P|}\int_P{\sum_{k=-\log_2{\ell(P)}}^{\infty}{\big|m_k^{\vee}\ast f_k(y) \big|^q}}dy \Big)^{1/q}}\Big\Vert_{L^p(x)}.
\end{equation*}
Now let  $x\in P\in \mathcal{D}_{\nu}$ for some $\nu\in\mathbb{Z}$ and define $P^*=9P$ as before. 
Using (\ref{recall1}), 
\begin{align*}
\Big(\frac{1}{|P|}\int_P{\sum_{k=\nu}^{\infty}\big|m_k^{\vee}\ast \big( \chi_{P^*}f_k\big)(x) \big|^q}dy \Big)^{1/q}&\lesssim  \mathcal{L}_{s}^{r}[\mm]\Big(\frac{1}{|P|}\int_{P^*}{\sum_{k=\nu}^{\infty}{\big(\mathfrak{M}_{\sigma,2^k}f_k(y) \big)^q}}dy\Big)^{1/q}\\
&\lesssim \mathcal{L}_{s}^{r}[\mm]\mathcal{M}_q\big(\big\Vert \{\mathfrak{M}_{\sigma,2^k}f_k(\cdot)\}_{k\in\zz}\big\Vert_{\ell^q} \big)(x) 
\end{align*} for $\sigma>d/q$.
Then the  $L^{p}$ boundedness of $\mathcal{M}_q$ and Peetre's maximal inequality (\ref{maximal1}) yield that
\begin{equation*}
\Big\Vert \sup_{P\in\mathcal{D}: x\in P}{\Big(\frac{1}{|P|}\int_P{\sum_{k=-\log_2{\ell(P)}}^{\infty}{\big|m_k^{\vee}\ast \big(\chi_{P^*}f_k\big)(y) \big|^q}}dy \Big)^{1/q}}\Big\Vert_{L^p(x)}\lesssim \mathcal{L}_{s}^{r}[\mm]\big\Vert \big\{ f_k\big\}_{k\in\mathbb{Z}}\big\Vert_{L^p(\ell^q)}.
\end{equation*}
Furthermore, it follows from (\ref{estvp}) that
\begin{align*}
\Big(\frac{1}{|P|}\int_P{\sum_{k=\nu}^{\infty}\big|m_k^{\vee}\ast \big( \chi_{(P^*)^c}f_k\big)(x) \big|^q}dy \Big)^{1/q}&\lesssim  \mathcal{L}_{s}^{r}[\mm] \Big(\frac{1}{|P|}\int_{P}{\big\Vert \{\mathfrak{M}_{\epsilon,2^k}f_k(y)\}_{k\in\zz}\big\Vert_{\ell^{\infty}}^q }dy\Big)^{1/q}\\
&\lesssim \mathcal{L}_{s}^{r}[\mm]\mathcal{M}_q\big(\big\Vert \{\mathfrak{M}_{\epsilon,2^k}f_k(\cdot)\}_{k\in\zz}\big\Vert_{\ell^{\infty}}\big)(x) .
\end{align*}
Then via the $L^{p}$ boundedness of $\mathcal{M}_q$, (\ref{maximal1}) with $\epsilon>d/p$, and the embedding $\ell^q\hookrightarrow \ell^{\infty}$ we have
\begin{equation*}
\Big\Vert \sup_{P\in\mathcal{D}: x\in P}{\Big(\frac{1}{|P|}\int_P{\sum_{k=-\log_2{\ell(P)}}^{\infty}{\big|m_k^{\vee}\ast \big(\chi_{(P^*)^c}f_k\big)(y) \big|^q}}dy \Big)^{1/q}}\Big\Vert_{L^p(x)}\lesssim\mathcal{L}_{s}^{r}[\mm]\big\Vert \big\{ f_k\big\}_{k\in\mathbb{Z}}\big\Vert_{L^p(\ell^q)}.
\end{equation*}
This proves that for $d/\epsilon<p<\infty$
\begin{equation}\label{q<t<p}
\big\Vert \big\{ m_k^{\vee}\ast f_k\big\}_{k\in\mathbb{Z}}\big\Vert_{L^p(\ell^q)}\lesssim \mathcal{L}_{s}^{r}[\mm] \Vert \big\{ f_k\big\}_{k\in\mathbb{Z}}\Vert_{L^p(\ell^q)}.
\end{equation} 

The general case $q<p<\infty$ follows from the interpolation method in Proposition \ref{complexinterpolation} between (\ref{q<t<p}) and $L^q(\ell^q)$ estimate with the same values of $s$ and $r$.

\subsection{\bf The case $1<p<\infty$ and $1<q\leq \infty$}
The proof is based on a suitable use of the complex interpolation method in Proposition \ref{complexinterpolation} and the duality property in Lemma \ref{dualitylemma}.

{\bf Step 1.} We claim that for $2< p<\infty$, $d/2-d/p=d/p'-d/2<s<d$, and $r>d/s$.
\begin{equation}\label{step1}
\big\Vert \big\{m_k^{\vee}\ast f_k\big\}_{k\in\zz}\big\Vert_{L^p(\ell^{p'})}\lesssim \mathcal{L}_{s}^{r}[\mm]\Vert \{f_k\}_{k\in\zz}\Vert_{L^p(\ell^{p'})}.
\end{equation}
Choose $\epsilon>0$ and $\wt{p}$ such that $s>d/r+\epsilon$ and $\max{(d/\epsilon,p)}<\wt{p}<\infty$.
Then, by using Lemma \ref{charactersharp} and the arguments used in obtaining (\ref{q<t<p}), we can prove that
\begin{equation*}
\big\Vert \big\{ m_k^{\vee}\ast f_k\big\}_{k\in\zz}\big\Vert_{L^{\wt{p}}(\ell^{p'})}\lesssim \mathcal{L}_{s}^{r}[\mm]\Vert \{f_k\}_{k\in\zz}\Vert_{L^{\wt{p}}(\ell^{p'})}.
\end{equation*}
Now (\ref{step1}) follows from the interpolation with the $L^{p'}(\ell^{p'})$ boundedness with the same values of $r$ and $s$ because $p'<p<\wt{p}$.

{\bf Step 2.} We prove that for $1< p<2$, $d/p-d/2=d/2-d/p'<s<d$, and $r>d/s$,
\begin{equation}\label{step2}
\big\Vert \big\{m_k^{\vee}\ast f_k\big\}_{k\in\zz}\big\Vert_{L^p(\ell^{p'})}\lesssim \mathcal{L}_{s}^{r}[\mm]\Vert \{f_k\}_{k\in\zz}\Vert_{L^p(\ell^{p'})}.
\end{equation}
Suppose that $\{f_k\}_{k\in\zz}\in L_A^{p}(\ell^{p'})$. 
By using Lemma \ref{dualitylemma}, the left-hand side of (\ref{step2}) can be dualized and estimated by
\begin{equation*}
\sup_{\{r_Q\}_{Q\in\mathcal{D}}:\Vert \{r_Q\}_{Q\in\mathcal{D}}\Vert_{\dot{f}_{p'}^{0,p}}\leq 1}{\Big|\int_{\rd}{\sum_{k\in\zz}{m_k^{\vee}\ast f_k(x)\mathfrak{V}_k^{\Psi_0}\big(\{r_Q\}_{Q\in\mathcal{D}}\big)(x)}}dx \Big|},
\end{equation*}
which can be also written as
\begin{equation*}
\sup_{\{r_Q\}_{Q\in\mathcal{D}}:\Vert \{r_Q\}_{Q\in\mathcal{D}}\Vert_{\dot{f}_{p'}^{0,p}}\leq 1}{\Big|\int_{\rd}{\sum_{k\in\zz}{ f_k(x)m_k^{\vee}\ast\big(\mathfrak{V}_k^{\Psi_0}\big(\{r_Q\}_{Q\in\mathcal{D}}\big)\big)(x)}}dx \Big|}.
\end{equation*}
This is clearly majorized, using H\"older's inequality, by
\begin{equation*}
\big\Vert \{f_k\}_{k\in\zz}\big\Vert_{L^{p}(\ell^{p'})}\sup_{\{r_Q\}_{Q\in\mathcal{D}}:\Vert \{r_Q\}_{Q\in\mathcal{D}}\Vert_{\dot{f}_{p'}^{0,p}}\leq 1}{\Big\Vert \Big\{m_k^{\vee}\ast \big(\mathfrak{V}_k^{\Psi_0}\big(\{r_Q\}_{Q\in\mathcal{D}}\big)\big) \Big\}_{k\in\zz}\Big\Vert_{L^{p'}(\ell^{p})}}.
\end{equation*} 
Moreover, the result in Step 1 and (\ref{discretecha2}) yield that the $L^{p'}(\ell^p)$-norm in the above expression is smaller than a constant times
\begin{equation*}
 \mathcal{L}_r^s[\mm]\big\Vert \big\{  \mathfrak{V}_k^{\Psi_0}\big(\{r_Q\}_{Q\in\mathcal{D}}\big)\big\}_{k\in\zz}\big\Vert_{L^{p'}(\ell^p)}\lesssim  \mathcal{L}_r^s[\mm]\Vert \{r_Q\}_{Q\in\mathcal{D}}\Vert_{\dot{f}_{p'}^{0,p}},
\end{equation*}
which proves (\ref{step2}).

{\bf Step 3.} 
Let $1<p<\infty$ and $q$ is between $p$ and $p'$ so that $|d/p-d/2|>|d/q-d/2|$. Suppose $|d/p-d/2|<s<d$ and $r>d/s$.
We interpolate two cases $(p,p')$ and $(p,p)$ by using Proposition \ref{complexinterpolation} with the same values of $s$ and $r$.
Then we establish the estimate
\begin{equation*}
\big\Vert \big\{m_k^{\vee}\ast f_k\big\}_{k\in\zz}\big\Vert_{L^p(\ell^{q})}\lesssim \mathcal{L}_{s}^{r}[\mm]\Vert \{f_k\}_{k\in\zz}\Vert_{L^p(\ell^{q})}.
\end{equation*}

{\bf Step 4.} 
Let $1<q<\infty$ and $p$ is between $q$ and $q'$ so that $|d/q-d/2|>|d/p-d/2|$. Suppose $|d/q-d/2|<s<d$ and $r>d/s$.
We interpolate two cases $(q',q)$ and $(q,q)$ by using Proposition \ref{complexinterpolation} with the same values of $s$ and $r$.
Then we have the estimate
\begin{equation*}
\big\Vert \big\{m_k^{\vee}\ast f_k\big\}_{k\in\zz}\big\Vert_{L^p(\ell^{q})}\lesssim \mathcal{L}_{s}^{r}[\mm]\Vert \{f_k\}_{k\in\zz}\Vert_{L^p(\ell^{q})}.
\end{equation*}

{\bf Step 5.} Let $1<p<\infty$ and $q=\infty$. Suppose $d/2<s<d$ and $r>d/s$.
An argument similar to that used in Step 2, with Lemma \ref{dualitylemma} and the result for $1<p<\infty$ and $q=1$, leads to the desired estimate. We skip the details to avoid unnecessary repetition.


\section{Proof of Theorem \ref{negativemain}}

We now describe the proof of Theorem \ref{negativemain}, using the ideas in \cite{Ch_Se, Gr_Park}.
Suppose $0<p<\infty$ or $p=q=\infty$.
\subsection{Necessary conditions for vector-valued operator inequalities} 
We investigate necessary conditions for the inequality that for $K\in\mathcal{E}(1)$, 
\begin{eqnarray}\label{boundassumption}
\big\Vert \big\{ 2^{kd}K(2^k\cdot)\ast f_k\big\}_{k\in\zz}\big\Vert_{L^p(\ell^q)} \leq \mathcal{A}\big\Vert \{f_k\}_{k\in\zz}\big\Vert_{L^p(\ell^q)}, \qquad f_k\in \mathcal{E}(2^{k-1})
\end{eqnarray}  for some $\mathcal{A}>0$.

An immediate consequence is that 
\begin{equation}\label{immediatecon}
\Vert K\Vert_{L^p(\rd)}\lesssim_{p} \mathcal{A},
\end{equation} which follows from setting $f_0=4^d\Psi_0(4\cdot)$ and $f_k=0$ for $k\not= 0$ so that 
\begin{equation*}
\Vert K \Vert_{L^p(\rd)}=\big\Vert \big\{ 2^{kd}K(2^k\cdot)\ast f_k\big\}_{k\in\zz}\big\Vert_{L^p(\ell^q)}\leq \mathcal{A}\Vert 4^d\Psi_0(4\cdot)\Vert_{L^p(\rd)}\lesssim \mathcal{A}.
\end{equation*}

Moreover,  it is known in \cite{Ch_Se} that if (\ref{boundassumption}) holds for $0<q\leq p<\infty$, then
\begin{eqnarray}\label{qbound}
\Vert K\Vert_{L^q(\rd)}\lesssim_{p,q} \mathcal{A}.
\end{eqnarray}

Now we consider the case $1<p,q<\infty$.
Using the dualization argument in Lemma \ref{dualitylemma}, which was used to obtain (\ref{step2}), the $L^p(\ell^q)$ boundedness also implies that 
\begin{align*}
&\Vert \{2^{kd}K(2^k\cdot)\ast f_k\}_{k\in\zz}\Vert_{L^{p'}(\ell^{q'})}\\
&\lesssim \big\Vert \{f_k\}_{k\in\zz}\big\Vert_{L^{p'}(\ell^{q'})}\sup_{\{r_Q\}_{Q\in\mathcal{D}}:\Vert \{r_Q\}_{Q\in\mathcal{D}}\Vert_{\dot{f}_{p}^{0,p}}\leq 1}{\Big\Vert \Big\{2^{kd}K(2^k\cdot)\ast \big(\mathfrak{V}_k^{\Psi_0}\big(\{r_Q\}_{Q\in\mathcal{D}}\big)\big) \Big\}_{k\in\zz}\Big\Vert_{L^{p}(\ell^{q})}}\\
&\lesssim \mathcal{A} \big\Vert \{f_k\}_{k\in\zz}\big\Vert_{L^{p'}(\ell^{q'})}.
\end{align*}
Therefore it is clear from (\ref{immediatecon}) that 
\begin{equation*}
\Vert K\Vert_{L^{p'}(\rd)}\lesssim_p \mathcal{A}
\end{equation*} and 
if $1<p\leq q< \infty$ ( that is, $1< q'\leq p'<\infty$ ), then
 we have  
\begin{equation*}
\Vert K\Vert_{L^{q'}(\rd)}\lesssim_{p,q} \mathcal{A}
\end{equation*}
from the estimate (\ref{qbound}).

We note that if $K\in\mathcal{E}(1)$, then Bernstein's inequality shows that 
\begin{equation}\label{kembedding}
\Vert K\Vert_{L^{r_1}(\rd)}\lesssim \Vert K\Vert_{L^{r_0}(\rd)} \qquad \text{ for }\quad r_0< r_1.
\end{equation}
Therefore, we conclude that 
\begin{lemma}
Let $0<p<\infty$ and $0<q\leq \infty$. Suppose that $K\in\mathcal{E}(1)$.
If (\ref{boundassumption}) holds, then
\begin{equation*}
\Vert K\Vert_{L^{\min{(p,q,p',q')}}(\rd)}\lesssim_{p,q,d}\mathcal{A}
\end{equation*} where we adhere to the standard convention that $p'=\infty$ for $p\leq 1$ and $q'=\infty$ for $q\leq 1$.
\end{lemma}

On the other hand, when $p\geq 1$,  (\ref{boundassumption}) implies that the convolution operator with $K$ is bounded in $L^p(\rd)$. Indeed, for any $f\in L^{p}(\rd)$ let 
\begin{equation*}
f_0:= 4^d\Psi_0(4\cdot) \ast f, \qquad \text{and }\qquad f_k:=0, ~~ k\not=0.
\end{equation*}
Then using the identity $K=4^d\Psi_0(4\cdot)\ast K$, we have
\begin{equation*}
\Vert K\ast f\Vert_{L^p(\rd)}=\big\Vert \{2^{kd}K(2^k\cdot)\ast f_k\}_{k\in\zz}\big\Vert_{L^p(l^q)}\leq \mathcal{A}\Vert f_0\Vert_{L^p(\rd)}\lesssim \mathcal{A}\Vert f\Vert_{L^p(\rd)}
\end{equation*} where the last inequality follows from Young's inequality with $p\geq 1$.
Hence it follows that
\begin{equation*}
\Vert \wh{K}\Vert_{L^{\infty}(\rd)}\lesssim \mathcal{A}.
\end{equation*}
By additionally assuming that $K\in\mathcal{E}(1)$ is a nonnegative function, we obtain that
\begin{equation*}
\Vert K\Vert_{L^1(\rd)}=\wh{K}(0)\leq \Vert \wh{K}\Vert_{L^{\infty}(\rd)}\lesssim \mathcal{A},
\end{equation*}
and this, together with (\ref{kembedding}), yields the following lemma:
\begin{lemma}\label{nececondition}
Let $0<p<\infty$ and $0<q\leq \infty$. Suppose that $K\in\mathcal{E}(1)$ is a nonnegative function on $\rd$.
If (\ref{boundassumption}) holds, then
\begin{equation*}
\Vert K\Vert_{L^{\min{(1,p,q)}}(\rd)}\lesssim_{p,q,d}\mathcal{A}.
\end{equation*}

\end{lemma}

\subsection{Construction of examples}
Note that $s<d/\min{(1,p,q)}$ implies $\min{(1,p,q)}<\tau^{(s,p,q)}$. 
Choosing
\begin{equation}\label{tgamma}
t:=\frac{d}{\min{(1,p,q)}} \qquad \text{ and }\qquad  \frac{2}{\tau^{(s,p,q)}}<\gamma<\frac{2}{\min{(1,p,q)}},
\end{equation}
we define
\begin{equation*}
\mathcal{H}^{(t,\gamma)}(x):=\dfrac{1}{(1+4\pi^2|x|^2)^{t/2 }}\dfrac{1}{(1+\ln(1+4\pi^2|x|^2))^{\gamma/2}}.
\end{equation*}
Then it is proved in \cite{Gr_Park} that
\begin{align}\label{esth}
\big|(I-\Delta)^{s/2}\wh{\mathcal{H}^{(t,\gamma)}}(\xi)\big|&=\big| \wh{\mathcal{H}^{(t-s,\gamma)}}(\xi)\big|\nonumber\\
&\lesssim_{t,\gamma,d}\begin{cases}
e^{-|\xi|/2} & \text{ for }~ |\xi|> 1\\
|\xi|^{-(d-t+s)}(1+2\ln{|\xi|^{-1}})^{-\gamma/2} &\text{ for }~ |\xi|\leq 1
\end{cases}
\end{align}
where $d-t+s=s-d/\min{(1,p,q)}+d>0$.

Let $\eta\in S(\rd)$ have the properties that
$\eta\geq 0$, $\eta(x)\geq c>0$ on $\{x\in\rd:|x|\leq 1/100\}$ for some $c>0$, and $Supp(\wh{\eta})\subset \{\xi\in\rd:|\xi|\leq 1/10\}$.
We define 
\begin{equation*}
K^{(t,\gamma)}(x):=\mathcal{H}^{(t,\gamma)}\ast \eta(x), \qquad K_k^{(t,\gamma)}(x):=2^{kd}K^{(t,\gamma)}(2^kx)
\end{equation*}
and \begin{equation*}
m_k^{(t,\gamma)}:= \wh{K_k^{(t,\gamma)}}.
\end{equation*}

We first observe that
\begin{equation*}
m_k^{(t,\gamma)}(2^k\xi)=\wh{K^{(t,\gamma)}}(\xi)=\wh{\mathcal{H}^{(t,\gamma)}}(\xi)\wh{\eta}(\xi)
\end{equation*}
and this yields that
\begin{align*}
\mathcal{L}_s^{\tau^{(s,p,q)}}[\mm]= \big\Vert \wh{\mathcal{H}^{(t,\gamma)}}\wh{\eta}\big\Vert_{L_s^{\tau^{(s,p,q)}}}\lesssim \big\Vert (I-\Delta)^{s/2}\wh{\mathcal{H}^{(t,\gamma)}}\big\Vert_{L^{\tau^{(x,p,q)}}(\rd)}
\end{align*}
where the Kato-Ponce inequality is applied.
Then using (\ref{esth}), we obtain that
\begin{equation*}
\mathcal{L}_s^{\tau^{(s,p,q)}}[\mm]\lesssim 1+\Big( \int_{|\xi|\leq 1}{\frac{1}{|\xi|^{\tau^{(s,p,q)}(d-t+s)}}\frac{1}{(1+2\ln{|\xi|^{-1}})^{\tau^{(s,p,q)}\gamma/2}}}d\xi\Big)^{1/\tau^{(s,p,q)}}
\end{equation*}
and using change of variables, the second term is estimated by a constant times
\begin{equation*}
\int_1^{\infty}{\frac{1}{u}\frac{1}{(1+2\ln{u})^{\tau^{(s,p,q)}\gamma/2}}}{du}<\infty
\end{equation*}
because $\tau^{(s,p,q)}(d-t+s)=d$ and $\tau^{(s,p,q)}\gamma/2>1$ with the choice of $t$ and $\gamma$ in (\ref{tgamma}).
Finally, we have 
\begin{equation*}
\mathcal{L}_s^{\tau^{(s,p,q)}}[\mm]\lesssim 1.
\end{equation*}

Now we suppose (\ref{mainest1}) holds with $m_k=m_k^{(t,\gamma)}$ and $A=2^{-2}$, which is equivalent to (\ref{boundassumption}) with $K=K^{(t,\gamma)}$ and $\mathcal{A}=\mathcal{L}_s^{\tau^{(s,p,q)}}[\mm]$.
Then it follows from Lemma \ref{nececondition} that
\begin{equation}\label{contradiction}
\Vert K^{(t,\gamma)}\Vert_{L^{\min{(1,p,q)}}(\rd)}\lesssim  \mathcal{L}_{s}^{\tau^{(s,p,q)}}[\mm]\lesssim 1.
\end{equation} 
 since $K^{(t,\gamma)}$ is a nonnegative function.
 However,
\begin{equation*}
\Vert K^{(t,\gamma)}\Vert_{L^{\min{(1,p,q)}}(\rd)}=\big\Vert \mathcal{H}^{(t,\gamma)}\ast \eta\big\Vert_{L^{\min{(1,p,q)}(\rd)}}\gtrsim \Vert \mathcal{H}^{(t,\gamma)}\Vert_{L^{\min{(1,p,q)}}}
\end{equation*} where the inequality follows from the fact that $\mathcal{H}^{(t,\gamma)}, \eta\geq 0$ and $\mathcal{H}^{(t,\gamma)}(x-y)\geq \mathcal{H}^{(t,\gamma)}(x)\mathcal{H}^{(t,\gamma)}(y)$.
This yields that
\begin{align*}
&\Vert K^{(t,\gamma)}\Vert_{L^{\min{(1,p,q)}}(\rd)}\\
&\gtrsim \Big( \int_{\rd}{\frac{1}{(1+4\pi^2|x|^2)^{d/2}}\frac{1}{(1+\ln(1+4\pi^2|x|^2))^{\gamma\min{(1,p,q)}/2}}}dx\Big)^{1/\min{(1,p,q)}}=\infty,
\end{align*}
since $\gamma\min{(1,p,q)/2}<1$,
which contradicts (\ref{contradiction}).

\appendix
\section{Proof of (\ref{membedding})}
(\ref{membedding}) is a consequence of the following lemma:
\begin{lemma}
Let $1<r_0<r_1<\infty$ and $s\geq 0$. Suppose that $f\in L_s^{r_1}(\rd)$ is supported in $\{x\in\rd:|x|\leq B \}$ for some $B>0$.
Then $f\in L_s^{r_0}(\rd)$ and indeed,
\begin{equation*}
\Vert f\Vert_{L_s^{r_0}(\rd)}\lesssim_s B^{d/r_0-d/r_1}\Vert f\Vert_{L_s^{r_1}(\rd)}.
\end{equation*}
\end{lemma}
\begin{proof}
Let $\Gamma\in S(\rd)$ satisfy $Supp(\Ga)\subset \{x\in\rd:|x|\leq 2B\}$ and $\Gamma(x)=1$ for $|x|\leq B$.
Define the multiplication operator $T$ by
\begin{equation*}
Tg(x):=g(x)\Gamma(x) \qquad \text{ for }~ g\in S(\rd).
\end{equation*}
Using H\"older's inequality and the Kato-Ponce inequality \cite{Ka_Po}, we obtain that for each $n\in\nn_0$,
\begin{equation*}
\Vert Tg\Vert_{L_n^{r_0}(\rd)}\lesssim B^{d/r_0-d/r_1}\Vert Tg\Vert_{L_n^{r_1}(\rd)}\lesssim_n B^{d/r_0-d/r_1}\Vert g\Vert_{L_n^{r_1}(\rd)}.
\end{equation*}
Then we interpolate these estimates to extend to 
\begin{equation}\label{extension}
\Vert Tg\Vert_{L_s^{r_0}(\rd)}\lesssim  B^{d/r_0-d/r_1}\Vert g\Vert_{L_s^{r_1}(\rd)}
\end{equation} for all $s\geq 0$.

Now suppose $g\in S(\rd)$ has compact support in $\{x\in \rd: |x|\leq B\}$ so that $g=Tg$. Then (\ref{extension}) implies that
\begin{equation*}
\Vert g\Vert_{L_s^{r_0}(\rd)}\lesssim B^{d/r_0-d/r_1}\Vert g\Vert_{L_s^{r_1}(\rd)},
\end{equation*}
from which the desired result follows, using the density of $S(\rd)$ in the two Banach spaces $L_s^{r_0}(\rd)$ and $L_s^{r_1}(\rd)$.

\end{proof}

\section*{Acknowledgement}

{Part of this research was carried out during my stay at the University of Missouri-Columbia. I would like to thank Professor L. Grafakos for his invitation, hospitality, and very useful discussions during the stay. I also would like to express gratitude to the anonymous referees for the careful reading and very useful comments.}

\end{document}